\documentclass[12pt]{amsart}
\usepackage{amsfonts,amsthm,amsmath,amscd,amssymb}
\usepackage[T1]{fontenc}
\usepackage[mathscr]{eucal}
\usepackage{indentfirst}
\usepackage{graphicx}
\usepackage{subcaption}

\usepackage{xcolor}
\usepackage{pict2e}
\usepackage{epic}
\usepackage{hyperref}
\numberwithin{equation}{section}
\usepackage[margin=2.9cm]{geometry}
\usepackage{epstopdf} 
\usepackage{cite}
\usepackage{calc}
\usepackage{parskip}
\usepackage{array}
\usepackage{booktabs}
\usepackage{enumerate}
\usepackage{adjustbox}
\usepackage{tikz}
\usetikzlibrary{calc}
\usepackage{listings}
\theoremstyle{plain}
\newtheorem{thm}{Theorem}[section]
\newtheorem{lem}[thm]{Lemma}
\newtheorem{cor}[thm]{Corollary}
\newtheorem{prop}[thm]{Proposition}

 \theoremstyle{definition}
\newtheorem{defn}[thm]{Definition}
\newtheorem{rem}[thm]{Remark}

\newtheorem{ques}[thm]{Question}
\newtheorem{setup}[thm]{Setup}

\newcommand{\ovl}{\overline}

\newcommand{\mb}[1]{\mathbb{#1}}
\newcommand{\mc}[1]{\mathcal{#1}}

\newcommand{\mr}[1]{\mathrm{#1}}

\newcommand{\Gal}{\operatorname{Gal}}
\newcommand{\Char}{\operatorname{char}}

\newcommand{\Spec}{\operatorname{Spec}}

\newcommand{\Tr}{\operatorname{Tr}}

\newcommand{\pcoor}[1]{%
  \begingroup\lccode`~=`: \lowercase{\endgroup
  \edef~}{\mathbin{\mathchar\the\mathcode`:}\nobreak}%
  [
  \begingroup
  \mathcode`:=\string"8000
  #1%
  \endgroup 
  ]
}

\begin{document}
\title{Rational lines on smooth cubic surfaces}

\author{Stephen McKean}
\address{Department of Mathematics \\ Brigham Young University} 
\email{mckean@math.byu.edu}
\urladdr{shmckean.github.io}


\begin{abstract}
We prove that the enumerative geometry of lines on smooth cubic surfaces is governed by the arithmetic of the base field. In 1949, Segre proved that the number of lines on a smooth cubic surface over any field is 0, 1, 2, 3, 5, 7, 9, 15, or 27. Over a given field, each of these line counts may or may not be realized by some cubic surface. We give a sufficient criterion for each of these line counts in terms of the Galois theory of the base field.
\end{abstract}

\maketitle

\section{Introduction}
In 1849, Cayley and Salmon proved that every smooth cubic surface over $\mb{C}$ contains exactly 27 complex lines~\cite{Cay49}. By 1858, Schl\"afli had proved that every smooth cubic surface over $\mb{R}$ contains exactly 3, 7, 15, or 27 real lines~\cite{Sch58}, with each of these counts occurring for some real cubic surface. Following this theme, B. Segre classified all possible rational line counts for smooth cubic surfaces over $\mb{Q}$ in 1949~\cite{Seg49}.

\begin{thm}[Segre]\label{thm:lines-over-Q}
Every smooth cubic surface over $\mb{Q}$ contains 0, 1, 2, 3, 5, 7, 9, 15, or 27 lines defined over $\mb{Q}$. Moreover, each of these counts is realized by some smooth cubic surface over $\mb{Q}$.
\end{thm}

B. Segre further showed in \textit{loc.~cit.}~that the line counts in Theorem~\ref{thm:lines-over-Q} are the only possible line counts for smooth cubic surfaces over any field:

\begin{thm}[Segre]\label{thm:lines-over-k}
The number of lines on a smooth cubic surface over any field must be $0,1,2,3,5,7,9,15,$ or $27$.
\end{thm}

In light of Theorem~\ref{thm:lines-over-k}, one can try to classify all which line counts actually occur for smooth cubic surfaces over a given field. These line counts will be a subset of $\{0,1,2,3,5,7,9,15,27\}$. For example, all line counts have been classified for smooth cubic surfaces over the following fields.

\begin{itemize}
\item Smooth cubic surfaces over $\mb{C}$ can only have 27 lines~\cite{Cay49}.
\item Smooth cubic surfaces over $\mb{R}$ can only have 3, 7, 15, or 27 lines, and each of these counts occurs~\cite{Sch58}.
\item Smooth cubic surfaces over $\mb{Q}$ can have 0, 1, 2, 3, 5, 7, 9, 15, or 27 lines, and each of these counts occurs~\cite{Seg49}.
\item Smooth cubic surfaces over $\mb{F}_2$ and $\mb{F}_3$ can only have 0, 1, 2, 3, 5, 9, or 15 lines, and each of these counts occurs~\cite{Dic15,LT19}.
\item Smooth cubic surfaces over $\mb{F}_5$ can only have 0, 1, 2, 3, 5, 7, 9, or 15 lines, and each of these counts occurs~\cite{LT19}.
\item Smooth cubic surfaces over $\mb{F}_q$ can have 0, 1, 2, 3, 5, 7, 9, 15, or 27 lines when $q>5$ is odd or $q=2^d$ with $d>1$, and each of these counts occurs~\cite{LT19}.
\end{itemize}

We clarify these results by providing, for each $n\in\{0,1,2,3,5,7,9,15,27\}$, a sufficient criterion for the occurrence of the line count $n$ over any given field $k$ (assuming $|k|\geq 23$). These criteria depend only on the Galois theory of $k$, so our main theorem can be summarized by saying that arithmetic governs the enumerative geometry of lines on cubic surfaces.

\begin{thm}\label{thm:main}
Let $k$ be a field with $|k|\geq 23$. There is a smooth cubic surface over $k$ whose 27 lines are all defined over $k$. Moreover, there is a smooth cubic surface over $k$ containing $n$ lines defined over $k$ if $k$ admits a separable field extension of the degrees listed in Table~\ref{table:main}.
\begin{table}[h]
\caption{Line counts and degrees of extensions}\label{table:main}
\renewcommand{\arraystretch}{1.6}
\aboverulesep=0ex
\belowrulesep=0ex
\centering
\begin{tabular}{c|c}
\toprule
$n$ & degree(s)\\
\toprule
15 & 2\\
9 & 3\\
7 & 2\\
5 & 4\\
\bottomrule
\end{tabular}
\qquad
\begin{tabular}{c|c}
\toprule
\makebox[\widthof{15}]{$n$} & degree(s)\\
\toprule
3 & 2\\
2 & 5\\
1 & 2 and 4\\
0 & 3 or 6\\
\bottomrule
\end{tabular}
\end{table}
\end{thm}

\begin{rem}\label{rem:thm about 3s}
Even more can be said about cubic surfaces with 3 lines. Such triples of lines are either skew or coplanar and pairwise intersecting. There is a smooth cubic surface over $k$ with 3 skew lines if $k$ admits separable extensions of degrees 2 and 3, and there is a smooth cubic surface over $k$ with 3 coplanar lines if $k$ admits a separable extension of degree 2.
\end{rem}

\begin{rem}
Theorem~\ref{thm:main} is only true if $|k|>5$, as the fields $\mb{F}_2$, $\mb{F}_3$, and $\mb{F}_5$ admit separable field extensions of arbitrary degree but do not admit all possible line counts for smooth cubic surfaces. While $\mb{F}_q$ does not contradict Theorem~\ref{thm:main} unless $q=2,3,5$, our methods do not account for this. We prove Theorem~\ref{thm:main} by blowing up Galois-invariant sets of points in the plane, and the cardinality assumption allows us to ensure that these sets of points can be arranged in general position.
\end{rem}

\begin{rem}
In an earlier version of this article, we claimed that the sufficient criteria in Theorem~\ref{thm:main} are also necessary. However, this claim was based on a mistake that was caught by Sam Streeter. We will point out this mistake later in the article. In forthcoming joint work with Kaya, Streeter, and Uppal \cite{KMSU25}, we give necessary criteria for line counts on cubic surfaces (and other del Pezzo surfaces) in terms of the arithmetic of the base field.
\end{rem}

\subsection{Application: classifying line counts}
As an application of Theorem~\ref{thm:main}, we classify all possible line counts for smooth cubic surfaces over finitely generated fields and finite transcendental extensions of arbitrary fields.

\begin{cor}\label{cor:fin-gen}
Let $k$ be a finitely generated field (with $|k|\geq 23$) or a finite transcendental extension of an arbitrary field. There is a smooth cubic surface over $k$ containing $n$ lines defined over $k$ for each $n\in\{0, 1, 2, 3, 5, 7, 9, 15, 27\}$.
\end{cor}

Loughran and Trepalin's classification of line counts over finite fields~\cite{LT19} also follows from Theorem~\ref{thm:main} (provided that the finite field contains at least 23 elements).

In \cite{KMSU25}, we give sufficient and necessary criteria for each line count, which can then be applied to characterize all line counts over a given field. For example, it holds that a smooth cubic surface over a real closed field must have 3, 7, 15, or 27 lines, generalizing Schl\"afli's classical count of lines on real cubic surfaces. One can also show that any smooth cubic surface over the field of complex constructible numbers has 0, 2, 5, 9, or 27 lines, as this field is quadratically closed.

\subsection{Remarks on the inverse Galois problem for cubic surfaces}\label{sec:igp}
B. Segre's proof of Theorem~\ref{thm:lines-over-k} is geometric. A modern approach to this theorem comes from the inverse Galois problem for cubic surfaces. An integer $n$ can be a line count for some smooth cubic surface over some field only if there is a subgroup conjugacy class of the Weyl group $W(\mathrm{E}_6)$ whose action on the Schl\"afli graph has $n$ fixed points. There are 25 conjugacy classes to consider, and each of the counts given in Theorem~\ref{thm:lines-over-k} occurs for at least one of these conjugacy classes. We include Loughran's Magma implementation of this computation in Appendix~\ref{sec:magma}. See also~\cite[Table 7.1]{BFL19} for a list of the conjugacy classes and their corresponding line counts.

Because we can deduce all occuring line counts over a given field $k$ by solving the inverse Galois problem for cubic surfaces over $k$, this inverse Galois problem is stronger than just classifying all line counts over $k$. The inverse Galois problem for cubic surfaces was solved over $\mb{Q}$ by Elsenhans and Jahnel~\cite{EJ15} (which thus gives an alternate proof of Theorem~\ref{thm:lines-over-Q}) and over finite fields by Loughran and Trepalin~\cite{LT19} (see also~\cite{BFL19}).

Loughran and Trepalin show that fewer conjugacy classes occur for cubic surfaces over $\mb{F}_2$ than over $\mb{F}_3$~\cite[Theorem 1.1]{LT19}, even though the sets of line counts over these two fields agree. In particular, the inverse Galois problem is \textit{strictly} stronger than classifying line counts. However, we can actually solve the inverse Galois problem for cubic surfaces over some fields by obstructing certain line counts. The only line counts coming from more than one conjugacy class are 0, 1, and 3, so if the only possible line counts over $k$ are a subset of $\{2,5,7,9,15,27\}$, then one can solve the inverse Galois problem for cubic surfaces over $k$ by characterizing line counts over $k$.

\subsection{Open question: counting lines on a given cubic surface}
In a slightly different direction, one can ask about the number of rational lines on a given cubic surface. Ideally, we would like to be able to determine this number directly from the defining polynomial of the given cubic surface. 

In joint work with Minahan and Zhang~\cite[Theorem 1.1]{MMZ20}, we proved that the number of real lines on a smooth cubic surface $X$ over $\mb{R}$ can be determined from the defining polynomial of $X$ and the defining equations of 3 skew real lines on $X$. It seems reasonable that one could generalize this result to hold over other subcomplex fields. 

The reason behind requiring the data of 3 skew lines is Galois-theoretic: the Galois group of solving for the 27 lines on a cubic surface is not solvable~\cite{Jor57}, so there is no equation in radicals for the defining equations of these 27 lines. In contrast, the Galois group of solving for the 27 lines on a cubic surface with 3 skew lines is solvable~\cite{Har79}, so there is a formula in radicals for the 27 lines in terms of the cubic surface and the given 3 skew lines. Even without the data of 3 skew lines, there is an algebraic function solving for the 27 lines on $X$ in terms of its defining polynomial, so one could hope that the count of rational lines on $X$ can also be read from its defining polynomial.

\begin{ques}
Let $k$ be a field. Given a homogeneous polynomial $F\in k[x_0,\ldots,x_3]$ of degree 3 whose associated cubic surface $\mb{V}(F)\subset\mb{P}^3$ is smooth, can the number of $k$-rational lines on $\mb{V}(F)$ be determined from the coefficients of $F$?
\end{ques}

\subsection{Methods and related work}
Studying rational lines on cubic surfaces via blow ups is a classical technique. See~\cite{LT19,BFL19} and the references therein for some recent applications of this approach. When blowing up collections of closed points to get smooth cubic surfaces, one technical requirement is that the points lie in general position. One can derive algebraic criteria for this by requiring the points to lie on the cuspidal cubic $\mb{V}(y^3-x^2z)\subset\mb{P}^2_k$ (see Section~\ref{sec:cusp}). We learned this trick from a private communication from J.-P. Serre, but the same idea appears in~\cite{PSS20}.

Building on an earlier version of this article, El Manssour--El Maazouz--Kaya--Rose use the method described in Section~\ref{sec:cusp} to show that all line counts occur for smooth cubic surfaces over $p$-adic fields~\cite{MMKR22}.

While we focus on the existence of line counts for cubic surfaces over certain fields, one can go further by investigating the distributions of these line counts or even classifying all cubic surfaces with a given line count. These distributions are known over finite fields due to the work of Das~\cite{Das20}. One can apply the methods of~\cite[Proposition 3.4]{PV02} to understand these distributions over $\mb{Q}$. There has been extensive work on the subject of classifying cubic surfaces and their lines over finite fields, especially on classifying cubic surfaces with 27 lines over a finite field. See e.g.~\cite{Hir67a,Hir67b,BHK18,BK19}.

\subsection{Outline and conventions}
We begin with an overview of some useful classical results in Section~\ref{sec:preliminaries}. We then give B. Segre's original geometric proof of Theorem~\ref{thm:lines-over-k} (with some details added and a minor error corrected) in Section~\ref{sec:proofs}. We prove Theorem~\ref{thm:main} in Section~\ref{sec:sufficient}. Finally, we apply Theorem~\ref{thm:main} in Section~\ref{sec:specific} to prove Corollary~\ref{cor:fin-gen}. In Appendix~\ref{sec:magma}, we give Loughran's code that gives a modern proof of Theorem~\ref{thm:lines-over-k}.

Throughout this article, we will only consider smooth cubic surfaces. When working over a field $k$, we will use the term \textit{rational lines} to refer to lines defined over $k$ (see Definition~\ref{defn:field of definition}). Whenever we write $Y\subseteq X$ or $Y\subset X$ for schemes $X,Y$, we mean that $Y$ is a closed subscheme of $X$.

\subsection*{Acknowledgements}
We thank Kirsten Wickelgren for her advice and support, as well as Alex Betts, Ronno Das, Igor Dolgachev, Enis Kaya, Viatcheslav Kharlamov, Aaron Landesman, Antonio Lerario, Dan Loughran, Jean-Pierre Serre, and Ravi Vakil for their helpful correspondence. We also thank the anonymous referee whose thorough feedback encouraged us to rewrite this article. We are especially indebted to J.W.P. Hirschfeld and J.-P. Serre for bringing Segre's result~\cite{Seg49} to our attention, to J.-P. Serre for several enlightening discussions, and to Sam Streeter for catching a serious error in a previous version of this article. We thank the anonymous referee whose thorough feedback inspired us to strengthen our results and rewrite this article. The author received support from an NSF MSPRF grant (DMS-2202825) and Kirsten Wickelgren's NSF CAREER grant (DMS-1552730).

\section{Preliminaries}\label{sec:preliminaries}
We state a few classical results that we will use throughout this article.

\begin{defn}\label{defn:field of definition}
Let $K/k$ be a field extension. We say that a closed subscheme $X\subseteq\mb{P}^n_{K}$ is \textit{defined over} $k$ or \textit{has field of definition} $k$ if the following equivalent conditions are satisfied (see e.g.~\cite[Proposition 1.2]{Dol16}).
\begin{enumerate}[(a)]
\item The defining ideal of $X$ is generated by homogeneous polynomials in $k[x_0,\ldots,x_n]$.
\item There exists a closed subscheme $Y\subseteq\mb{P}^n_k$ such that $X=Y\times_{\Spec{k}}\Spec{K}$.
\end{enumerate}
\end{defn}

Any closed subscheme of projective space has a minimal field of definition by~\cite[IV\textsubscript{2}, Corollaire (4.8.11)]{EGA}. If a scheme $X$ has field of definition $k$, we may also say that $X$ is \textit{$k$-rational}. Definition~\ref{defn:field of definition} (b) immediately implies that field of definition is preserved under base change.

\begin{prop}\label{prop:field under base change}
Let $k\subseteq K\subseteq K'$ be a tower of fields. Let $X\subseteq\mb{P}^n_K$ be a closed subscheme. If $X$ is defined over $k$, then the base change $X_{K'}=X\times_{\Spec{K}}\Spec{K'}$ is defined over $k$.
\end{prop}
\begin{proof}
By assumption, there exists a closed subscheme $Y\subseteq\mb{P}^n_k$ such that $X=Y\times_{\Spec{k}}\Spec{K}$. Thus
\begin{align*}
X_{K'}&=(Y\times_{\Spec{k}}\Spec{K})\times_{\Spec{K}}\Spec{K'}\\
&=Y\times_{\Spec{k}}\Spec{K'},
\end{align*}
as desired.
\end{proof}

Since closed immersions are stable under base change~\cite[\href{https://stacks.math.columbia.edu/tag/01JY}{Lemma 01JY}]{Sta18}, Proposition~\ref{prop:field under base change} states that $k$-rational subschemes get sent to $k$-rational subschemes under base change. The converse is also true.

\begin{prop}\label{prop:descend field of definition}
Let $k\subseteq K\subseteq K'$ be a tower of fields. Let $X,Y\subseteq\mb{P}^n_K$. Suppose that $Y_{K'}\subseteq X_{K'}$ and that $X_{K'},Y_{K'}$ are both defined over $k$. Then $Y\subseteq X$, and $X,Y$ are both defined over $k$.
\end{prop}
\begin{proof}[Proof of Proposition~\ref{prop:descend field of definition}]
Field extensions are fpqc and closed immersions satisfy fpqc descent~\cite[IV\textsubscript{2}, Proposition (2.7.1) (xii)]{EGA}, so the assumption that $Y_{K'}\subseteq X_{K'}$ implies that $Y\subseteq X$. 

We now show that $X$ is defined over $k$. The proof that $Y$ is defined over $k$ follows the same argument. Let $\mc{I}$ and $\mc{J}$ be the defining ideals of $X$ and $X_{K'}$, respectively, so that $\mc{I}=\mc{J}\cap K[x_0,\ldots,x_n]$. Under Definition~\ref{defn:field of definition} (a), the assumption that $X_{K'}$ is defined over $k$ means that there are homogeneous polynomials $f_1,\ldots,f_m\in k[x_0,\ldots,x_n]$ such that $\mc{J}=(f_1,\ldots,f_m)\cdot K'[x_0,\ldots,x_m]$. Since $k[x_0,\ldots,x_n]\subseteq K[x_0,\ldots,x_n]$, it follows that $\mc{J}\cap K[x_0,\ldots,x_n]$ is again generated by $f_1,\ldots,f_m$. In particular, $\mc{I}$ is generated by homogeneous polynomials in $k[x_0,\ldots,x_n]$, so $X$ is defined over $k$.
\end{proof}

For any field extension $K/k$, a cubic surface over $k$ is smooth if and only if its base change to $K$ is smooth (see e.g. \cite[IV\textsubscript{4}, Proposition (17.3.3) (iii) and Corollaire (17.7.3) (ii)]{EGA}). Together with Propositions~\ref{prop:field under base change} and~\ref{prop:descend field of definition}, this means that we can enumerate $k$-rational lines on $X$ by base changing to a field $K$ over which all 27 lines on $X$ are defined and studying the $k$-rationality of lines on $X_K$. Since smooth cubic surfaces are separably split~\cite{Coo88}, all lines on a smooth cubic surface over a field $k$ are defined over the separable closure $k^s$ (within any chosen algebraic closure of $k$).

We have thus reduced the study of rational lines on $X$ to the study of $k$-rational lines on $X_{k^s}$. We will study the field of definition of lines on cubic surfaces by acting on the relevant varieties by the absolute Galois group. This was done classically for lines on cubic surfaces over $\mb{R}$, as well as by Pannekoek~\cite{Pan09} for studying Galois orbits of lines on cubic surfaces over number fields. 

\begin{prop}\label{prop:Galois-inv}
Let $k$ be a field, and fix a separable closure $k^s$ of $k$. A geometrically reduced closed subscheme $X\subseteq\mb{P}^n_{k^s}$ is defined over $k$ if and only if $\sigma\cdot X=X$ for all $\sigma\in\Gal(k^s/k)$.
\end{prop}
\begin{proof}
The group $\Gal(k^s/k)$ acts on the defining ideal $\mc{I}\subseteq k^s[x_0,\ldots,x_n]$ of $X$ by acting on the coefficients of each $f\in\mc{I}$. If $X$ is defined over $k$, then the coefficients of any generating set of $\mc{I}$ are fixed under $\Gal(k^s/k)$-action and hence so is $X$. 

Now suppose $X$ is fixed under $\Gal(k^s/k)$-action. By Hilbert's Basis Theorem, $X$ is defined by a finite set $\{f_1,\ldots,f_r\}$ of polynomials over some finite extension $k'\subseteq k^s$ of $k$. Given $f\in\mc{I}$ and $\sigma\in\Gal(k'/k)$, denote the image of $f$ under $\sigma$-action by $f^\sigma$. Since $\sigma\cdot X=X$, we have that $f^\sigma(p)=0$ for all $p\in X$. In particular, $f^\sigma\in\mc{I}$ for all $f\in\mc{I}$. The desired result follows from~\cite[Lemma 1 (b)]{HRC12}. We describe the relevant ideas here. Fix a $k$-basis $\{e_1,\ldots,e_m\}$ of $k'$, and let $\Tr_{k'/k}:k'[x_0,\ldots,x_n]\to k[x_0,\ldots,x_n]$ be given by taking the Galois trace of each coefficient of a given polynomial. Then $\{\Tr_{k'/k}(e_if_j)\}_{i,j}$ generates the ideal $\mc{I}$. Moreover, since $\Tr_{k'/k}(e_if_j)^\sigma=\Tr_{k'/k}(e_if_j)$ for all $\sigma\in\Gal(k'/k)$, it follows that $\Tr_{k'/k}(e_if_j)\in k[x_0,\ldots,x_n]$. Thus $\mc{I}$ is generated by polynomials over $k$, as desired.
\end{proof}

A cubic surface $X$ defined over $k$ is fixed by $\Gal(k^s/k)$-action, so Galois action preserves the set of 27 lines on $X_{k^s}$. Moreover, Galois action preserves the incidence relations of the 27 lines:

\begin{prop}\label{prop:defined}
Let $k$ be a field with $k^s$ a fixed separable closure, $X$ be a smooth cubic surface defined over $k$, and $\sigma\in\Gal(k^s/k)$. Two lines $L$ and $L'$ in $X_{k^s}$ intersect if and only if $\sigma\cdot L$ and $\sigma\cdot L'$ intersect.
\end{prop}
\begin{proof}
The $\sigma$-action is defined pointwise. In particular, if $L$ and $L'$ intersect in the point $p$, then $\sigma\cdot L$ and $\sigma\cdot L'$ intersect in the point $\sigma\cdot p$. Conversely, if $\sigma\cdot L$ and $\sigma\cdot L'$ intersect in the point $q$, then $L$ and $L'$ intersect in the point $\sigma^{-1}\cdot q$.
\end{proof}

\begin{prop}\label{prop:residual-is-rationally-determined}
Let $k$ be a field with fixed separable closure $k^s$, and let $X$ be a smooth cubic surface defined over $k$. If $L_1,L_2,L_3\subseteq X_{k^s}$ are three coplanar lines, and if $L_1$ and $L_2$ are defined over $k$, then $L_3$ is also defined over $k$.
\end{prop}
\begin{proof}
Since $L_1$ and $L_2$ are defined over $k$, the plane $H\subset\mb{P}^3_k$ that contains them is also defined over $k$. By B\'ezout's theorem (and the fact that all lines on $X$ are defined over $k^s$~\cite{Coo88}), we have $H_{k^s}\cap X_{k^s}=L_1\cup L_2\cup L_3$. The varieties $L_1,L_2,H$, and $X$ are each fixed by all $\Gal(k^s/k)$-actions since they are defined over $k$. We now act on the configuration $H_{k^s}\cap X_{k^s}$ by each $\sigma\in\Gal(k^s/k)$. Since $H$ and $X$ are defined over $k$, we have $\sigma\cdot(H_{k^s}\cap X_{k^s})=H_{k^s}\cap X_{k^s}$. That is, $L_1\cup L_2\cup L_3=(\sigma\cdot L_1)\cup(\sigma\cdot L_2)\cup(\sigma\cdot L_3)$. Since $L_1$ and $L_2$ are defined over $k$, we have $\sigma\cdot L_1=L_1$ and $\sigma\cdot L_2=L_2$, so $L_1\cup L_2\cup L_3=L_1\cup L_2\cup(\sigma\cdot L_3)$. It follows that $L_3=\sigma\cdot L_3$ for all $\sigma\in\Gal(k^s/k)$, so $L_3$ is defined over $k$.
\end{proof}

\begin{cor}\label{cor:two-lines-give-third}
If a smooth cubic surface $X$ over a field $k$ contains two rational lines $L_1,L_2$ that intersect each other, then $X$ contains a third rational line $L_3$ that intersects $L_1$ and $L_2$.
\end{cor}
\begin{proof}
Let $H$ be the plane containing $L_1$ and $L_2$. By B\'ezout's theorem and~\cite{Coo88}, $X_{k^s}\cap H_{k^s}$ consists of (the base changes of) $L_1,L_2$, and a third line $L_3$. Since $L_1$ and $L_2$ are defined over $k$, Proposition~\ref{prop:residual-is-rationally-determined} implies that $L_3$ is also defined over $k$. Thus each of these three lines on $X_{k^s}$ are the base change of a $k$-rational line on $X$, and their intersection data are preserved by Proposition~\ref{prop:defined}.
\end{proof}

It is a classical result that every smooth cubic surface is the blow-up of $\mb{P}^2$ at 6 general points --- provided that one works over an algebraically closed field. In general, a smooth cubic surface need not be birational to $\mb{P}^2$. For example, Schl\"afli proved that there are smooth cubic surfaces over $\mb{R}$ whose $\mb{R}$-points are homeomorphic to $\mb{RP}^2\sqcup S^2$, where $S^2$ is a 2-sphere (for a modern treatment, see e.g.~\cite[Section 5]{Kol97}). So while a smooth cubic surface $X$ over an arbitrary field $k$ need not be rational, $X$ is \textit{geometrically} rational: we can view $X_{\ovl{k}}$ as the blow-up of $\mb{P}^2_{\ovl{k}}$ at 6 points. In fact, a result of Coombes \cite{Coo88} implies that every cubic surface is \textit{separably} rational.

\begin{lem}\label{lem:separably rational}
Let $k$ be a field, and let $k^s$ be the separable closure of $k$ in some algebraic closure $\ovl{k}$. If $X$ is a smooth cubic surface over $k$, then $X_{k^s}$ is the blow-up of $\mb{P}^2_{k^s}$ at 6 points.
\end{lem}
\begin{proof}
The proof follows a classical argument. We assume that $k=k^s$ (to simplify notation), so that all 27 lines on $X$ are $k$-rational~\cite{Coo88}. Since $X$ contains at least four rational lines, B\'ezout's theorem implies that $X$ contains two skew rational lines (otherwise all four lines would be coplanar, contradicting the fact that $X$ is cubic). We can also show that each line on $X$ meets 5 pairs of intersecting lines on $X$, with each pair of lines disjoint from the others~\cite[Chapter IV.2.5, p.~256]{Sha13}. 

Let $L\subset X$ be a line, and let $\{L_i,L_i'\}_{i=1}^5$ be the set of pairs of lines meeting $L$ (with $L_i\cap L_i'\neq\varnothing$ and $(L_i\cup L_i')\cap(L_j\cup L_j')=\varnothing$ for $i\neq j$). If $\Lambda\subset X$ is a line that does not meet $L$, then $\Lambda$ meets at most one of $L_i,L_i'$ for each $i$ (otherwise $L$ and $\Lambda$ would be coplanar and hence not disjoint). In fact, $\Lambda$ meets precisely one of $L_i,L_i'$ for each $i$. To see this, let $H_i$ be the plane such that $X\cap H_i=L\cup L_i\cup L_i'$. Since $\Lambda\subset X$, the intersection $H_i\cap\Lambda$ consists of a single point that must lie on $X$. Thus $H_i\cap\Lambda\subset X\cap H_i=L\cup L_i\cup L_i'$. Since $L\cap\Lambda=\varnothing$ by assumption, we are done and can conclude that there are exactly 5 lines in $X$ meeting any skew pair of lines.

Given two skew lines $L_1,L_2\subset X$, we construct mutually inverse rational maps $\phi:X\dashrightarrow L_1\times L_2$ and $\psi:L_1\times L_2\dashrightarrow X$ as follows. For each $x\in X\backslash(L_1\cup L_2)$, let $L_x\subset\mb{P}^3$ be the unique line through $x$ and meeting $L_1$ and $L_2$. Define $\phi(x)=(L_1\cap L_x,L_2\cap L_x)$. For each $(\ell_1,\ell_2)\in L_1\times L_2$, let $\ovl{\ell_1\ell_2}$ be the line through $\ell_1$ and $\ell_2$. If $\ovl{\ell_1\ell_2}$ is not contained in $X$, then B\'ezout's theorem implies that $X\cap\ovl{\ell_1\ell_2}$ consists of three distinct points: $L_1\cap\ovl{\ell_1\ell_2}$, $L_2\cap\ovl{\ell_1\ell_2}$, and a third point, which we denote $\psi(\ell_1,\ell_2)$. Thus $X$ is birational to $L_1\times L_2\cong\mb{P}^1_k\times\mb{P}^1_k$.

We next extend $\phi:X\dashrightarrow\mb{P}^1_k\times\mb{P}^1_k$ to a morphism. If $x\in X\backslash L_i$, let $H_i$ be the unique plane in $\mb{P}^3_k$ containing $L_i\cup x$ for $i=1,2$. If $x\in L_i$, let $H_i=T_xX$. Setting $\phi(x)=(H_2\cap L_1,H_1\cap L_2)$, one can check that $\phi:X\to\mb{P}^1_k\times\mb{P}^1_k$ is now a well-defined morphism. The inverse of $\phi$ is not well-defined precisely where $\psi$ is not well-defined, namely whenever $\ovl{\ell_1\ell_2}\subset X$. These are lines in $X$ that meet the two skew lines $L_1$ and $L_2$, and there are 5 such lines. One then checks that $\phi:X\to\mb{P}^1_k\times\mb{P}^1_k$ is a blow-up at these 5 points. Since $\mb{P}^1_k\times\mb{P}^1_k$ is the blow-up of $\mb{P}^2_k$ at 1 point, it follows that $X$ is the blow-up of $\mb{P}^2_k$ at 6 points.
\end{proof}

\begin{rem}
As mentioned in the introduction, we claimed in a previous version of this article that the sufficient criteria listed in Theorem~\ref{thm:main} are also necessary. This was based on an incorrect application of Lemma~\ref{lem:separably rational} --- we claimed that if $X_{k^s}$ is $k$-rational, then the blowdown locus of 6 points in $\mb{P}^2_{k^s}$ must also be $k$-rational. However, as Sam Streeter pointed out to us, this is not true. Given a closed embedding $Z\subseteq Y$, the blowup $\mr{Bl}_Z(Y)$ can be $k$-rational even if $Z$ is not $k$-rational. For example, there are real cubic surfaces that are not birational to $\mb{RP}^2$ (Schl\"afli's fifth type), whereas these cubic surfaces are geometrically rational (as are all smooth cubic surfaces).
\end{rem}

\section{The list of possible line counts}\label{sec:proofs}
We now give B. Segre's proof of Theorem~\ref{thm:lines-over-k}. The general idea is to pass to the separable closure, work geometrically, and keep track of the field of definition of each line. As mentioned in Section~\ref{sec:igp}, Theorem~\ref{thm:lines-over-k} can be proved by a computation on the Weyl group $W(\mathrm{E}_6)$ (see Appendix~\ref{sec:magma} for a Magma implementation of this computation, provided to us by Loughran). However, we find Segre's geometric proof interesting and worth expositing. We will add various details omitted from Segre's original account, streamline some of the arguments, and correct Segre's unnecessary $\Char{k}\neq 2$ assumption.

The proof utilizes a few geometric facts, which we list for the reader's convenience. These facts are classical, although we keep track of the field of definition of the lines involved when necessary. We will omit any proofs that do not require us to keep track of fields of definition. The first fact is that every line $L$ on a smooth cubic surface meets exactly one line in each triple of coplanar lines (to which $L$ does not belong).

\begin{lem}\label{lem:each line meets triangle}
Let $X$ be a smooth cubic surface. Given three pairwise-intersecting lines $L_1,L_2,L_3$ on $X$, any other line on $X$ meets exactly one of $L_1,L_2,L_3$.
\end{lem}

In Lemma~\ref{lem:four skew gives 15}, we will show that if a smooth cubic surface $X$ contains four skew $k$-rational lines, then $X$ contains either 15 or 27 $k$-rational lines. Two key ingredients are: each triple of skew lines on $X$ meets a unique triple of skew lines on $X$, and each quadruple of skew lines on $X$ meets a unique pair of skew lines on $X$.

\begin{prop}\label{prop:three skew meet three skew}
Let $X$ be a smooth cubic surface. Given three skew lines $L_1,L_2,L_3\subset X$, there is a unique triple $M_1,M_2,M_3\subset X$ of skew lines that each meet $L_i$.
\end{prop}

\begin{prop}\label{prop:four skew gives two}
Let $X$ be a smooth cubic surface. Given four skew lines $L_1,\ldots,L_4\subset X$, there is a unique pair $L,L'\subset X$ of skew lines meeting each $L_i$. 
\end{prop}

We will also need the facts that each pair of skew lines on $X$ belongs to a unique double six that splits the pair, and that the intersection graph of the 15 lines in the complement of any double six is given by Figure~\ref{fig:fifteen}.

\begin{defn}
A \textit{double six} is a collection $\{L_i,L_i'\}_{i=1}^6$ of twelve lines such that $L_1,\ldots,L_6$ are skew, $L_1',\ldots,L_6'$ are skew, $L_i$ and $L_i'$ are skew, and $L_i$ and $L_j'$ are not skew for $i\neq j$. The two subsets $\{L_i\}$ and $\{L_i'\}$ are called \textit{sextuples}.
\end{defn}

\begin{prop}\label{prop:double six}
Let $X$ be a smooth cubic surface. Given two skew lines $L,L'\subset X$, there is a unique double six of lines on $X$ with $L$ and $L'$ belonging to different sextuples.
\end{prop}

\begin{lem}\label{lem:intersection graph of 15}
Let $X$ be a smooth cubic surface. The intersection graph of the lines in the complement of any double six on $X$ is the graph given in Figure~\ref{fig:fifteen}.
\end{lem}
\begin{proof}
 The incidence pattern of the lines in the complement of a double six, together with the tritangent planes to which they belong, form the Cremona--Richmond configuration $\mathrm{CR}$~\cite{Sch58}. Since the Cremona--Richmond configuration is self-dual, we can take the vertices of $\mathrm{CR}$ to represent the 15 lines on $X$ and the lines of $\mathrm{CR}$ to represent the tritangent planes to which these lines on $X$ belong. Each such tritangent plane corresponds to a 3-cycle in $G$ since coplanar lines on $X$ are pairwise-intersecting. It follows that we can obtain $G$ by ``projectivizing'' $\mathrm{CR}$: we turn each line on $\mathrm{CR}$ into a 3-cycle by joining the vertices on each end with a new edge. This is the graph given in Figure~\ref{fig:fifteen}.
\end{proof}

We can now show that if a smooth cubic surface $X$ contains four skew $k$-rational lines, then $X$ contains either 15 or 27 $k$-rational lines.

\begin{lem}\label{lem:four skew gives 15}
Let $X$ be a smooth cubic surface over a field $k$ with four skew $k$-rational lines $L_1,\ldots,L_4\subset X$. Let $L,L'\subset X$ be the (not necessarily $k$-rational) lines meeting each $L_i$. Let $D$ be the double six of $X$ such that $L$ and $L'$ belong to different sextuples. The 15 lines of $X$ not belonging to $D$ are all defined over $k$, and the lines belonging to $D$ are either all defined over $k$ or all not defined over $k$.
\end{lem}
\begin{proof}
We will first show that $L,L'$ are either both $k$-rational or both not $k$-rational. Three skew lines in $\mb{P}^3_k$ determine a unique quadric surface. Let $Q$ be the quadric determined by $L_1,L_2,L_3$. Let $M_1,M_2,M_3$ be the triple of skew lines meeting $L_1,L_2,L_3$ as given by Proposition~\ref{prop:three skew meet three skew}. B\'ezout's theorem and the fact that each $M_i$ meets each $L_1,L_2,L_3$ implies that $M_1,M_2,M_3$ are also contained in $Q$. Since $L$ and $L'$ meet $L_1,L_2,L_3$, we deduce that $L,L'\in\{M_1,M_2,M_3\}$. Because the set $\{L_1,L_2,L_3\}$ is Galois-fixed, the quadric $Q$ and both of its rulings are all defined over $k$. To solve for $L,L'$, we first compute the intersection $Q\cap L_4=\{p_1,p_2\}$. We then take the ruling $R$ of $Q$ that does not contain $L_1,L_2,L_3$ and find the lines $R_1,R_2\in R$ that pass through $p_1,p_2$, respectively. Algebraically, this corresponds to solving a quadratic equation over $k$. Since the roots of a quadratic equation over $k$ have the same field of definition, $L$ and $L'$ are either both $k$-rational or both not $k$-rational.

Fix a separable closure $k^s$ of $k$, and let $G_k=\Gal(k^s/k)$. Let $\Lambda$ be such that $\{L,L',\Lambda\}=\{M_1,M_2,M_3\}$. While $L,L'$ need not be $k$-rational, the line $\Lambda$ is $k$-rational. Indeed, $L_4$ and $Q$ are defined over $k$, so the intersection $L_4\cap Q$ is fixed under $G_k$-action. Both rulings of $Q$ are also defined over $k$, so the set of lines through $L_4\cap Q$ in either of these rulings is fixed under $G_k$. Thus $\{L,L'\}$ is $G_k$-fixed. Since $X$ is also defined over $k$, the intersection $X\cap Q=\{L_1,L_2,L_3,L,L',\Lambda\}$ is $G_k$-fixed. Since $L_1,L_2,L_3$ are all $k$-rational and $\{L,L'\}$ is $G_k$-fixed, it follows that $\Lambda$ is also $G_k$-fixed and is hence $k$-rational by Proposition~\ref{prop:Galois-inv}.

Since $L$ and $L'$ belong to different sextuples in $D$, any line in $D$ must be skew to exactly one of $L$ and $L'$. In particular, the lines $L_1,\ldots,L_4$ do not belong to $D$. Since $\Lambda$ is in a different ruling of $Q$ than $L_1,L_2,L_3$, the rational lines $\Lambda$ and $L_i$ intersect and hence determine a new $k$-rational line $N_i\subset X$ for $1\leq i\leq 3$. Note that each $N_i$ cannot meet $L$ or $L'$, or else we would have two distinct triples of coplanar lines that both contain $L_i$ and $N_i$. In particular, $N_i\not\in D$ for each $i$. 

The line $L_4$ does not meet $L_1,L_2,L_3$, or $\Lambda$, so Lemma~\ref{lem:each line meets triangle} implies that $L_4$ meets each $N_i$ for $1\leq i\leq 3$. We thus obtain new $k$-rational lines $P_1,P_2,P_3$, with $P_i$ meeting $L_4$ and $N_i$. As with each $N_i$, the fact that $L_4$ meets $L$ and $L'$ implies that $P_i$ cannot meet $L$ or $L'$, so $P_i\not\in D$. We have thus found 11 $k$-rational lines in the complement of $D$.

We now fill out the rest of the intersection graph of the complement of $D$. Each $P_i$ must be adjacent to the 3-cycle $\{\Lambda,L_j,N_j\}$ for $i\neq j$. In order to avoid creating two distinct 3-cycles that share an edge, $P_i$ cannot intersect $\Lambda$ or $N_j$. We thus get a $k$-rational line $A_{i,j}$ meeting $P_i$ and $L_j$. Since $A_{i,j}$ meets $L_j$ and $L_j$ meets $L,L'$, it follows as before that $A_{i,j}\not\in D$. Moreover, working within the graph in Figure~\ref{fig:fifteen} shows that $A_{i,j}=A_{j,i}$, so we have found 14 $k$-rational lines in the complement of $D$. The final line in the complement of $D$ is residual to $N_\ell$ and $A_{i,j}$ (where $\{i,j,\ell\}=\{1,2,3\}$), so the final line in $D$ is $k$-rational as well.

Let $S$ and $S'$ be the sextuples in $D$ to which $L$ and $L'$ respectively belong. Since $L$ and $L'$ are disjoint, $L$ intersects each line in $S'-\{L'\}$. Let $\Lambda\in S'-\{L'\}$ be such a line. There is a third line $R\subset X$ that intersects both $L$ and $\Lambda$. Since $R$ intersects $L$, we have $R\not\in S$. Since $R$ intersects $\Lambda\in S'$, we have $R\not\in S'$. Thus $R$ is not contained in the double six $D$, so $R$ is $k$-rational by the previous paragraph. If $L$ is $k$-rational, then Corollary~\ref{cor:two-lines-give-third} implies that $\Lambda$ is also $k$-rational. Similarly, if $L$ is not $k$-rational, then we deduce that $\Lambda$ cannot be $k$-rational by the contraposition of Corollary~\ref{cor:two-lines-give-third}. Repeating this argument for all lines in $S'-\{L'\}$, as well as the symmetric argument for all lines in $S-\{L\}$, we find that $X$ contains exactly 15 $k$-rational if $L,L'$ are not $k$-rational or 27 $k$-rational lines if $L,L'$ are $k$-rational.
\end{proof}

The final fact we will use is that if a smooth cubic surface $X$ contains two triples of coplanar $k$-rational lines, then $X$ contains a \textit{Steiner system} of $k$-rational lines.

\begin{defn}
A set $\{L_i^j\}_{i,j=1}^3$ of nine lines on a smooth cubic surface is called a \textit{Steiner system} if $L_i^1,L_i^2,L_i^3$ are coplanar for all $i$ and $L_1^j,L_2^j,L_3^j$ are coplanar for all $j$.
\end{defn}

\begin{lem}\label{lem:steiner}
Let $X$ be a smooth cubic surface over a field $k$. Let $L_1^1,L_2^1,L_3^1$ and $L_1^2,L_2^2,L_3^2$ be two distinct triples of $k$-rational coplanar lines on $X$. Then there exist $k$-rational lines $L_1^3,L_2^3,L_3^3\subset X$ such that $\{L_i^j\}_{i,j=1}^3$ form a Steiner system.
\end{lem}
\begin{proof}
Lines on $X$ intersect if and only if they are coplanar. Thus Lemma~\ref{lem:each line meets triangle} implies that for each $1\leq i\leq 3$, the line $L_i^1$ meets $L_j^2$ for precisely one of $1\leq j\leq 3$. Symmetrically, the line $L_i^2$ meets $L_j^1$ for precisely one of $1\leq j\leq 3$. Thus the lines $L_i^1$ and $L_j^2$ can be paired off into three couples of intersecting lines. Relabel the lines $L_j^2$ so that $L_i^1\cap L_i^2\neq\varnothing$ for each $i$. Since all of the lines at hand are $k$-rational, each of these pairs gives rise to another $k$-rational line by Corollary~\ref{cor:two-lines-give-third}. Denote the new $k$-rational line coming from $L_i^1$ and $L_i^2$ by $L_i^3$. Then $\{L_i^j\}_{i,j=1}^3$ is the desired Steiner system.
\end{proof}

We are almost ready to prove Theorem~\ref{thm:lines-over-k}. We will phrase our argument in terms of the \textit{intersection graph} $G$ of $X$. The vertices of $G$ correspond to $k$-rational lines on $X$, and two vertices of $G$ are adjacent if the corresponding lines on $X$ intersect. We will reinterpret some of the above geometric facts in terms of the intersection graph, after which we will prove Theorem~\ref{thm:lines-over-k}.

\begin{lem}\label{lem:intersection graph}
Let $X$ be a smooth cubic surface over a field $k$, and let $G$ be its intersection graph. Then:
\begin{enumerate}[(i)]
\item Every edge of $G$ belongs to a 3-cycle.
\item No two 3-cycles in $G$ share an edge.
\item If $G$ contains a 3-cycle, then $G$ is connected.
\item If $G$ contains two 3-cycles that do not share any vertices, then $G$ contains a Steiner system.
\item If $G$ contains four non-adjacent vertices, then $G$ is either the graph given in Figure~\ref{fig:fifteen} or the complement of the Schl\"afli graph.
\end{enumerate}
\end{lem}
\begin{proof}
An edge in $G$ corresponds to two intersecting $k$-rational lines, and a 3-cycle in $G$ corresponds to three pairwise-intersecting (equivalently, coplanar) $k$-rational lines. Thus Corollary~\ref{cor:two-lines-give-third} implies (i).

If $G$ were to contain two 3-cycles that shared an edge, then the four vertices in this configuration would correspond to four coplanar lines on $X$. Letting $H$ be the plane containing these four lines, we would have four lines in $X\cap H$. But B\'ezout's theorem implies that $X\cap H$ contains at most $\deg{X}\cdot\deg{H}=3$ lines, so we deduce (ii) by contradiction.

Lemma~\ref{lem:each line meets triangle} implies that if $G$ contains a 3-cycle, then every vertex in $G$ is adjacent to one of the vertices in the 3-cycle. This in turn implies that $G$ is connected, giving us (iii).

Item (iv) is just a restatement of Lemma~\ref{lem:steiner}. Item (v) follows from Lemmas~\ref{lem:intersection graph of 15} and~\ref{lem:four skew gives 15}. Indeed, four non-adjacent vertices in $G$ correspond to four skew lines on $X$. The $k$-rational lines on $X$ then either belong to the complement of a double six (whose intersection graph is given in Figure~\ref{fig:fifteen}), or all 27 lines on $X$ are $k$-rational (whose intersection graph is the complement of the Schl\"afli graph~\cite{Sch58,Tod32}).
\end{proof}

\begin{proof}[Proof of Theorem~\ref{thm:lines-over-k}]
The method of proof is to list all graphs satisfying the criteria given in Lemma~\ref{lem:intersection graph}. Let $G$ be the intersection graph of a smooth cubic surface over a field $k$. There are no obstructions to $G$ being the empty graph (Figure~\ref{fig:zero}), a single vertex (Figure~\ref{fig:one}), or two disjoint vertices (Figure~\ref{fig:two}). If $G$ contains an edge between two vertices, then $G$ contains a 3-cycle by Lemma~\ref{lem:intersection graph} (i). There are no obstructions to $G$ consisting of three vertices with no edges (Figure~\ref{fig:three skew}) or a 3-cycle (Figure~\ref{fig:three coplanar}).

If $G$ contains at least four vertices, then every vertex of $G$ must belong to a 3-cycle. Indeed, if $G$ contains an edge and therefore a 3-cycle by Lemma~\ref{lem:intersection graph} (i), then $G$ is connected by Lemma~\ref{lem:intersection graph} (iii). It follows that every vertex of $G$ has an incident edge and therefore belongs to a 3-cycle. If $G$ contains four disjoint vertices, then $G$ contains an edge by Lemma~\ref{lem:intersection graph} (v) and hence every vertex of $G$ belongs to a 3-cycle.

It follows that if $G$ contains at least four vertices, then we must obtain $G$ by taking a 3-cycle $C$, adjoining additional 3-cycles to $C$ (with each additional 3-cycle meeting $C$ at precisely one of its vertices), and then adding any edges necessary to satisfy the constraints listed in Lemma~\ref{lem:intersection graph}. One consequence is that if $G$ contains at least four vertices, then $G$ must contain an odd number of vertices, since we begin with a 3-cycle but only add two new vertices for each additional 3-cycle.

There is only one way to construct a graph with five vertices in this manner (Figure~\ref{fig:five}). To obtain a graph with seven vertices, we take two adjoined 3-cycles and adjoin a third 3-cycle. If these three 3-cycles do not share a common vertex, then we have a chain of 3-cycles (see Figure~\ref{fig:seven chain}). Each vertex on one end of this chain must be adjacent to the 3-cycle on the other end, so we need to add edges accordingly (one example illustrated in cyan). We must add more edges to $G$ until every edge belongs to a 3-cycle, but this will then force $G$ to contain two 3-cycles that share an edge (one example illustrated in red). This contradicts Lemma~\ref{lem:intersection graph} (ii), so we conclude that all three 3-cycles must be joined at a common vertex (Figure~\ref{fig:seven}).

If $G$ has nine vertices, then we attach three 3-cycles $T_1,T_2,T_3$ to a fourth 3-cycle $C$. If two of $T_1,T_2,T_3$ are attached to $C$ at the same vertex, then there are four non-adjacent vertices in $G$ (see Figure~\ref{fig:nine skew}). We must therefore add edges until no set of four vertices are mutually non-adjacent, but any choice of such edges will conflict with Lemma~\ref{lem:intersection graph} (ii). We therefore conclude that each $T_i$ must be adjoined to $C$ at a different vertex of $C$. We must then add edges until $G$ satisfies Lemma~\ref{lem:intersection graph}. This process results in the intersection graph of the Steiner system (Figure~\ref{fig:nine}).

To conclude, we will show that if $G$ has more than nine vertices, then $G$ contains four non-adjacent vertices. If $G$ has more than nine vertices, then $G$ is obtained by attaching at least four 3-cycles to a central 3-cycle. By the pigeonhole principle, $G$ will contain one of the graphs in Figure~\ref{fig:nine skew} as a subgraph. If no four vertices of $G$ are non-adjacent, we will need to add edges to $G$, but we have already seen that this will force $G$ to conflict with Lemma~\ref{lem:intersection graph} (ii). We thus conclude that $G$ contains four non-adjacent vertices, so $G$ has 15 or 27 vertices by Lemma~\ref{lem:intersection graph} (v).
\end{proof}

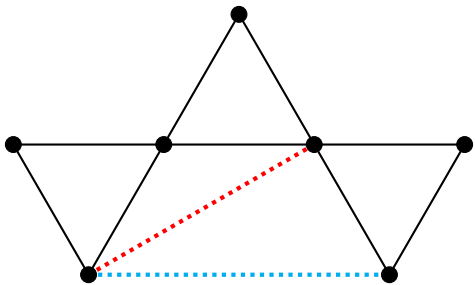
\begin{figure}[t]
\centering
\begin{tikzpicture}
\foreach [count=\i] \coord in {(-1,{-sqrt(3)/2}),(1,{-sqrt(3)/2}),(0,{sqrt(3)/2}),(-2,{-3*sqrt(3)/2}),(-3,{-sqrt(3)/2}),(3,{-sqrt(3)/2}),(2,{-3*sqrt(3)/2})}{
	\node[coordinate] (\i) at \coord {};}
	\draw[thick] (1) -- (2) -- (3) -- (1);
	\draw[thick] (1) -- (4) -- (5) -- (1);
	\draw[thick] (2) -- (6) -- (7) -- (2);
	\draw[ultra thick,cyan,dotted] (4) -- (7);
	\draw[ultra thick,red,dotted] (4) -- (2);
\foreach \i in {1,...,7}{
	\draw[fill=black] (\i) circle (3pt);
	}
\end{tikzpicture}
\caption{Impermissible graph of seven lines}\label{fig:seven chain}
\end{figure}

\begin{figure}[h]
\centering
\begin{subfigure}[c]{.3\linewidth}
\centering
\begin{tikzpicture}
\draw[white!0] (-2,-2) rectangle (2,2);
\foreach \i in {1,...,8}{
	\node[coordinate] (\i) at (-22.5+\i*45:1.5) {};
	\draw[fill=black] (\i) circle (3pt);
	}
	\node[coordinate] (0) at (0,0) {};
	\draw[fill=black] (0) circle (3pt);
	\draw[thick] (0) -- (1) -- (2) -- (0);
	\draw[thick] (0) -- (3) -- (4) -- (0);
	\draw[thick] (0) -- (5) -- (6) -- (0);
	\draw[thick] (0) -- (7) -- (8) -- (0);
	
\foreach \i in {1,3,5,7}{
	\draw[red,fill=red] (\i) circle (3pt);
	}
\end{tikzpicture}
\caption{All $T_i$ sharing a vertex}
\end{subfigure}
\begin{subfigure}[c]{.3\linewidth}
\centering
\begin{tikzpicture}
\draw[white!0] (-2,-2) rectangle (2,2);
\begin{scope}[shift={(0,{sqrt(3)/4})}]
\foreach [count=\i] \coord in {(0,0),(-1,0),(-1/2,{-sqrt(3)/2}),(1,0),(1/2,{-sqrt(3)/2}),(-1/2,{sqrt(3)/2}),(1/2,{sqrt(3)/2}),(0,{-sqrt(3)}),(-1,{-sqrt(3)})}{
	\node[coordinate] (\i) at \coord {};
	\draw[fill=black] (\i) circle (3pt);
	}
	\draw[thick] (1) -- (2) -- (3) -- (1);
	\draw[thick] (1) -- (4) -- (5) -- (1);
	\draw[thick] (1) -- (6) -- (7) -- (1);
	\draw[thick] (3) -- (8) -- (9) -- (3);
	\draw[red,fill=red] (8) circle (3pt);
	\draw[red,fill=red] (2) circle (3pt);
	\draw[red,fill=red] (4) circle (3pt);
	\draw[red,fill=red] (6) circle (3pt);
\end{scope}
\end{tikzpicture}
\caption{Two $T_i$ sharing a vertex}
\end{subfigure}
\caption{Impermissible graphs of nine lines}\label{fig:nine skew}
\end{figure}
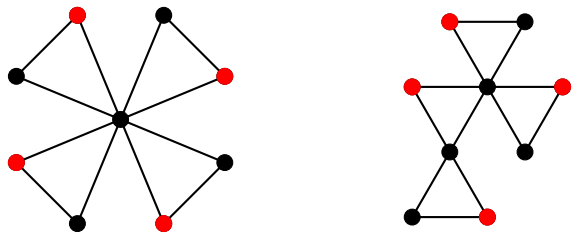

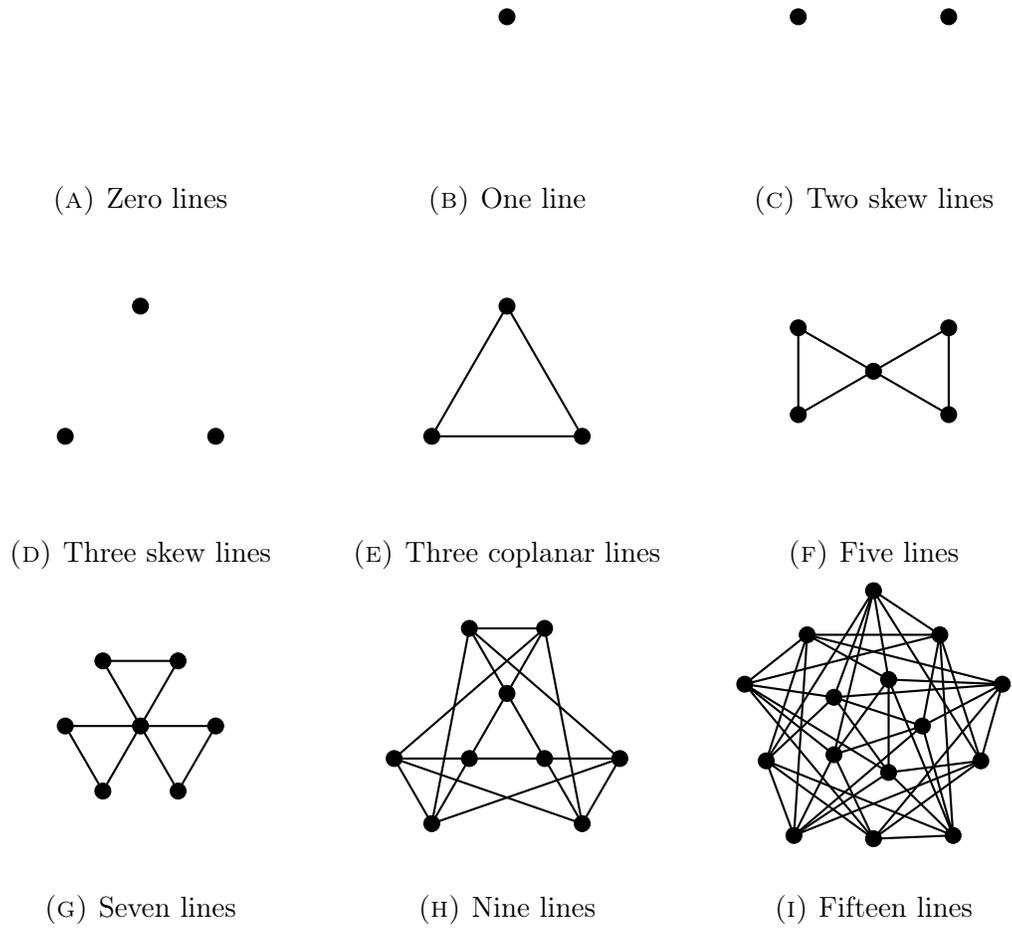
\begin{figure}[p]
\centering
\begin{subfigure}[c]{.3\linewidth}
\centering
\begin{tikzpicture}
\draw[white!0] (-2,-2) rectangle (2,2);
\end{tikzpicture}
\caption{Zero lines}\label{fig:zero}
\end{subfigure}
\begin{subfigure}[c]{.3\linewidth}
\centering
\begin{tikzpicture}
\draw[white!0] (-2,-2) rectangle (2,2);
\draw[fill=black] (0,0) circle (3pt);
\end{tikzpicture}
\caption{One line}\label{fig:one}
\end{subfigure}
\begin{subfigure}[c]{.3\linewidth}
\centering
\begin{tikzpicture}
\draw[white!0] (-2,-2) rectangle (2,2);
\draw[fill=black] (1,0) circle (3pt);
\draw[fill=black] (-1,0) circle (3pt);
\end{tikzpicture}
\caption{Two skew lines}\label{fig:two}
\end{subfigure}

\begin{subfigure}[c]{.3\linewidth}
\centering
\begin{tikzpicture}
\draw[white!0] (-2,-2) rectangle (2,2);
\foreach [count=\i] \coord in {(-1,{-sqrt(3)/2}),(1,{-sqrt(3)/2}),(0,{sqrt(3)/2})}{
	\draw[fill=black] \coord circle (3pt);
	}
\end{tikzpicture}
\caption{Three skew lines}\label{fig:three skew}
\end{subfigure}
\begin{subfigure}[c]{.3\linewidth}
\centering
\begin{tikzpicture}
\draw[white!0] (-2,-2) rectangle (2,2);
\foreach [count=\i] \coord in {(-1,{-sqrt(3)/2}),(1,{-sqrt(3)/2}),(0,{sqrt(3)/2})}{
	\node[coordinate] (\i) at \coord {};
	\draw[fill=black] (\i) circle (3pt);
	}
	\draw[thick] (1) -- (2) -- (3) -- (1);
\end{tikzpicture}
\caption{Three coplanar lines}\label{fig:three coplanar}
\end{subfigure}
\begin{subfigure}[c]{.3\linewidth}
\centering
\begin{tikzpicture}
\draw[white!0] (-2,-2) rectangle (2,2);
\foreach [count=\i] \coord in {(0,0),(-1,{1/sqrt(3)}),(-1,{-1/sqrt(3)}),(1,{1/sqrt(3)}),(1,{-1/sqrt(3)})}{
	\node[coordinate] (\i) at \coord {};
	\draw[fill=black] (\i) circle (3pt);
	}
	\draw[thick] (1) -- (2) -- (3) -- (1);
	\draw[thick] (1) -- (4) -- (5) -- (1);
\end{tikzpicture}
\caption{Five lines}\label{fig:five}
\end{subfigure}

\begin{subfigure}[c]{.3\linewidth}
\centering
\begin{tikzpicture}
\draw[white!0] (-2,-2) rectangle (2,2);
\foreach [count=\i] \coord in {(0,0),(-1,0),(-1/2,{-sqrt(3)/2}),(1,0),(1/2,{-sqrt(3)/2}),(-1/2,{sqrt(3)/2}),(1/2,{sqrt(3)/2})}{
	\node[coordinate] (\i) at \coord {};
	\draw[fill=black] (\i) circle (3pt);
	}
	\draw[thick] (1) -- (2) -- (3) -- (1);
	\draw[thick] (1) -- (4) -- (5) -- (1);
	\draw[thick] (1) -- (6) -- (7) -- (1);
\end{tikzpicture}
\caption{Seven lines}\label{fig:seven}
\end{subfigure}
\begin{subfigure}[c]{.3\linewidth}
\centering
\begin{tikzpicture}
\draw[white!0] (-2,-2) rectangle (2,2);
\foreach [count=\i] \coord in {(-1/2,{-sqrt(3)/4}),(1/2,{-sqrt(3)/4}),(0,{sqrt(3)/4}),(-1,{-3*sqrt(3)/4}),(-3/2,{-sqrt(3)/4}),(3/2,{-sqrt(3)/4}),(1,{-3*sqrt(3)/4}),(-1/2,{3*sqrt(3)/4}),(1/2,{3*sqrt(3)/4})}{
	\node[coordinate] (\i) at \coord {};
	\draw[fill=black] (\i) circle (3pt);
	}
	\draw[thick] (1) -- (2) -- (3) -- (1);
	\draw[thick] (1) -- (4) -- (5) -- (1);
	\draw[thick] (2) -- (6) -- (7) -- (2);
	\draw[thick] (3) -- (8) -- (9) -- (3);
	\draw[thick] (4) -- (6) -- (8) -- (4);
	\draw[thick] (5) -- (7) -- (9) -- (5);
\end{tikzpicture}
\caption{Nine lines}\label{fig:nine}
\end{subfigure}
\begin{subfigure}[c]{.3\linewidth}
\centering
\begin{tikzpicture}
\draw[white!0] (-2,-2) rectangle (2,2);
\foreach \i in {1,...,5}{
	\node[coordinate] (c\i) at (\i*72:.65) {};
	\node[coordinate] (m\i) at (-18+\i*72:1.5) {};
	\node[coordinate] (o\i) at (18+\i*72:1.8) {};
	}
	\draw[thick] (c1) -- (m2) -- (o5) -- (c1);
	\draw[thick] (c2) -- (m3) -- (o1) -- (c2);
	\draw[thick] (c3) -- (m4) -- (o2) -- (c3);
	\draw[thick] (c4) -- (m5) -- (o3) -- (c4);
	\draw[thick] (c5) -- (m1) -- (o4) -- (c5);
	\draw[thick] (c1) -- (c3) -- (o1) -- (c1);
	\draw[thick] (c2) -- (c4) -- (o2) -- (c2);
	\draw[thick] (c3) -- (c5) -- (o3) -- (c3);
	\draw[thick] (c4) -- (c1) -- (o4) -- (c4);
	\draw[thick] (c5) -- (c2) -- (o5) -- (c5);
	\draw[thick] (m1) -- (m2) -- (o2) -- (m1);
	\draw[thick] (m2) -- (m3) -- (o3) -- (m2);
	\draw[thick] (m3) -- (m4) -- (o4) -- (m3);
	\draw[thick] (m4) -- (m5) -- (o5) -- (m4);
	\draw[thick] (m5) -- (m1) -- (o1) -- (m5);
	
\foreach \i in {1,...,5}{
	\draw[fill=black] (c\i) circle (3pt);
	\draw[fill=black] (m\i) circle (3pt);
	\draw[fill=black] (o\i) circle (3pt);
	}
\end{tikzpicture}
\caption{Fifteen lines}\label{fig:fifteen}
\end{subfigure}
\caption{Intersection graphs}
\label{fig:intersection graphs}
\end{figure}

\subsection{Characteristic 2}
While his proof works in any characteristic, B. Segre incorrectly states that Theorem~\ref{thm:lines-over-k} fails in characteristic 2. He then proceeds to describe three smooth cubic surfaces over $\mb{F}_2$ that contain 35, 13, and 6 lines. These line counts contradict the classification of smooth cubic surfaces over $\mb{F}_2$ given by Dickson~\cite{Dic15}. In private communication, J.-P. Serre pointed out to us that Segre's lines are defined set-wise rather than algebraically. That is, Segre implicitly defines a line $L$ to be contained in a smooth cubic surface $X$ if every rational point of $L$ is contained in $X$. Since $\mb{P}^3_{\mb{F}_2}$ contains 15 rational points and 35 rational lines, Segre calculates the lines in his examples by checking which of these 15 points are contained in his cubic surfaces.

Over fields of cardinality at least 3, the set-theoretic and algebraic definitions of line containment for cubic surfaces agree.

\begin{prop}\label{prop:set-vs-algebraic}
Let $k$ be a field of cardinality at least 3. Let $X$ be a cubic surface defined over $k$. Let $L$ be a line defined over $k$. Then $L$ is contained in $X$ if and only if every $k$-rational point of $L$ is contained in $X$.
\end{prop}
\begin{proof}
If $L\subset X$, then every point of $L$ is contained in $X$. Thus all $k$-rational points of $L$ are contained in $X$. Conversely, suppose all $k$-rational points of $L$ are contained in $X$. Since $L$ is defined over $k$, this line is isomorphic to $\mb{P}^1_k$. If $|k|>2$, then $L\cong\mb{P}^1_k$ contains $|k|+1>3$ points defined over $k$. Since $\deg{L}\cdot\deg{X}=3$, B\'ezout's theorem implies that $L$ must be contained in $X$.
\end{proof}

Proposition~\ref{prop:set-vs-algebraic} fails over $\mb{F}_2$. Indeed, $\mb{P}^1_{\mb{F}_2}$ only contains three $\mb{F}_2$-rational points, so a cubic surface $X$ may intersect a line $L$ in three $\mb{F}_2$-points without containing any points of $L$ not defined over $\mb{F}_2$. This accounts for the discrepancy between Segre's claim and Dickson's theorem~\cite{Dic15} about lines on cubic surfaces over $\mb{F}_2$.

\section{Sufficient criteria for line counts}\label{sec:sufficient}
In this section, we prove Theorem~\ref{thm:main} by blowing up $\mb{P}^2_k$ at appropriate sets of points. The various line counts arise from the arithmetic configurations of the sets at which we blow up. This technique is classical and appears extensively in the study of lines on real cubic surfaces. As we will see in Section~\ref{sec:cusp}, most of the work goes into showing that the relevant sets of points can be arranged in general position over $k$.

\begin{setup}\label{setup:blowing up}
There is a well-known method for constructing the 27 lines on a smooth cubic surface via blow-ups. If $X$ is the smooth cubic surface obtained by blowing up $\mb{P}^2_k$ at the six geometric points $p_1,\ldots,p_6$, then we get the following lines on $X$:
\begin{itemize}
\item $E_i$, the exceptional divisor above $p_i$, for $1\leq i\leq 6$.
\item $C_i$, the strict transform of the unique conic through $\{p_1,\ldots,p_6\}-\{p_i\}$, for $1\leq i\leq 6$.
\item $L_{ij}$, the strict transform of the unique line through $p_i$ and $p_j$, for $1\leq i<j\leq 6$.
\end{itemize}
We also deduce that only pairs among these lines that intersect are $E_i$ and $C_j$ for $i\neq j$, $E_n$ and $L_{ij}$ for $n=i,j$, $C_n$ and $L_{ij}$ for $n=i,j$, and $L_{ij}$ and $L_{mn}$ for $\{i,j\}\cap\{m,n\}=\varnothing$.

Since all lines on $X$ are defined over a separable closure $k^s$ of $k$~\cite{Coo88}, we may take $p_1,\ldots,p_6$ to be $k^s$-rational. We can now check whether $X$ and any of its line are defined over $k$ using Proposition~\ref{prop:Galois-inv}. The surface $X$ is defined over $k$ if and only if the set $\{p_1,\ldots,p_6\}$ is $\Gal(k^s/k)$-fixed. Moreover:
\begin{itemize}
\item $E_i$ is $k$-rational if and only if $p_i$ is $k$-rational.
\item $C_i$ is $k$-rational if and only if the set $\{p_1,\ldots,p_6\}-\{p_i\}$ is $\Gal(k^s/k)$-fixed.
\item $L_{ij}$ is $k$-rational if and only if the set $\{p_i,p_j\}$ is $\Gal(k^s/k)$-fixed.
\end{itemize}
Each possible line count (and configuration) for smooth cubic surfaces arises from partitioning $\{p_1,\ldots,p_6\}$ into Galois orbits.
\begin{itemize}
\item[$(1_6)$] If $p_1,\ldots,p_6$ are all $k$-rational, then all 27 lines on $X$ are $k$-rational.
\item[$(1_4,2_1)$] If $p_1,\ldots,p_4$ are $k$-rational and $\{p_5,p_6\}$ is a Galois orbit, then $L_{56}$ and $E_i,C_j,L_{ij}$ for $1\leq i,j\leq 4$ are the only $k$-rational lines on $X$. Thus $X$ has 15 $k$-rational lines.
\item[$(1_3,3_1)$] If $p_1,p_2,p_3$ are $k$-rational and $\{p_4,p_5,p_6\}$ is a Galois orbit, then $E_i,C_j,L_{ij}$ for $1\leq i,j\leq 3$ are the only $k$-rational lines on $X$. Thus $X$ has 9 $k$-rational lines.
\item[$(1_2,2_2)$] If $p_1,p_2$ are $k$-rational, $\{p_3,p_4\}$ is a Galois orbit, and $\{p_5,p_6\}$ is a Galois orbit, then $E_1,E_2,C_1,C_2,L_{12},L_{34}$, and $L_{56}$ are the only $k$-rational lines on $X$. Thus $X$ has 7 $k$-rational lines.
\item[$(1_2,4_1)$] If $p_1,p_2$ are $k$-rational and $\{p_3,\ldots,p_6\}$ is a Galois orbit, then $E_1,E_2,C_1,C_2$, and $L_{12}$ are the only $k$-rational lines on $X$. Thus $X$ has 5 $k$-rational lines.
\item[$(1_1,2_1,3_1)$] If $p_1$ is $k$-rational, $\{p_2,p_3\}$ is a Galois orbit, and $\{p_4,p_5,p_6\}$ is a Galois orbit, then $E_1,C_1$, and $L_{23}$ are the only $k$-rational lines on $X$. Thus $X$ contains 3 skew $k$-rational lines.
\item[$(2_3)$] If $\{p_i,p_j\}$ is a Galois orbit for $(i,j)=(1,2)$, $(3,4)$, and $(5,6)$, then $L_{12}$, $L_{34}$, and $L_{56}$ are the only $k$-rational lines on $X$. Thus $X$ contains 3 $k$-rational lines that are pairwise intersecting.
\item[$(1_1,5_1)$] If $p_1$ is $k$-rational and $\{p_2,\ldots,p_6\}$ is a Galois orbit, then $E_1$ and $C_1$ are the only $k$-rational lines on $X$. Thus $X$ contains 2 $k$-rational lines.
\item[$(2_1,4_1)$] If $\{p_1,p_2\}$ and $\{p_3,\ldots,p_6\}$ are each Galois orbits, then $L_{12}$ is the only $k$-rational line on $X$. Thus $X$ contains 1 $k$-rational line.
\item[$(3_2)$] If $\{p_1,p_2,p_3\}$ and $\{p_4,p_5,p_6\}$ are each Galois orbits, then $X$ has no $k$-rational lines.
\item[$(6_1)$] If $\{p_1,\ldots,p_6\}$ is a Galois orbit, then $X$ has no $k$-rational lines.
\end{itemize}
\end{setup}

\subsection{Blowing up on the cuspidal cubic}\label{sec:cusp}
In order to obtain a smooth cubic surface by blowing up $\mb{P}^2_k$ at one of the sets described in Setup~\ref{setup:blowing up}, we need to ensure that the points at which we are blowing up lie in general position. (Over $\ovl{k}$, our collection of points splits into six points. These six points are said to lie in \textit{general position} if no three are contained in a line and all six are not contained in a conic.)

If we require our six points to lie on the cuspidal cubic $C=\mb{V}(y^3-x^2z)$, the parameterization $C=\{\pcoor{1:t:t^3}\}$ gives us an algebraic method for checking whether the points lie in general position. Three distinct points $\pcoor{1:t_i:t_i^3}$ lie on the line $\mb{V}(ax+by+cz)$ if and only if each $t_i$ is a root of $F(t)=a+bt+ct^3$. The sum of these roots is a scalar multiple of the coefficient of the degree 2 term of $F(t)$, so the points $\pcoor{1:t_i:t_i^3}$ lie on a shared line if and only if $t_1+t_2+t_3=0$. Similarly, six distinct points $\pcoor{1:t_i:t_i^3}$ lie on the conic $\mb{V}(ax^2+bxy+cy^2+xz+eyz+fz^2)$ if and only if each $t_i$ is a root of $G(t)=a+bt+ct^2+dt^3+et^4+ft^6$. The sum of these roots is a scalar multiple of the coefficient of the degree 5 term of $G(t)$, so $\pcoor{1:t_i:t_i^3}$ lie on a shared conic if and only if $t_1+\ldots+t_6=0$.

In order to find six points on the cuspidal cubic that lie in general position in $\mb{P}^2_k$, it therefore suffices to construct a degree 6 monic polynomial $G(t)$ such that:
\begin{enumerate}[(i)]
\item $G(t)$ has no repeated roots.
\item No three roots of $G(t)$ sum to zero.
\item The degree 5 coefficient of $G(t)$ is not zero.
\end{enumerate}

In Section~\ref{sec:proof of sufficient}, we will prove Theorem~\ref{thm:main} by constructing various degree 6 monic polynomials that satisfy the above criteria.

\subsection{Proof of Theorem~\ref{thm:main}}\label{sec:proof of sufficient}
We are now ready to prove Theorem~\ref{thm:main}. We treat each Galois partition of $\{p_1,\ldots,p_6\}$ (see Setup~\ref{setup:blowing up}) in a separate lemma. We remark that the cardinality assumptions on $k$ in each of these lemmas need not be optimal --- for example, Lemma~\ref{lem:27 lines} requires $|k|\geq 13$ in order to find 6 $k$-rational points in general position on the cuspidal cubic. If one does not restrict to the cusp, then there are collections of 6 $k$-rational points in $\mb{P}^2_k$ in general position. However, we will be content to restrict our search to points on the cuspidal cubic. We also remark that the proofs in this section are fairly computational. This is intentional --- in case some reader wishes to construct a smooth cubic surface over a given field with a desired line count, these proofs outline how to find an explicit set of points at which to blow up.

\begin{rem}\label{rem:not 0}
When constructing a degree 6 polynomial with at least one rational root, we will always avoid taking 0 to be one of our roots. In effect, this means that none of our 6 points in the plane will lie on the cusp point of the cuspidal cubic. The reason for this is to allow us to apply the Cayley--Bacharach theorem when constructing sets of 8 points in general position in \cite{KMSU25}.
\end{rem}

Throughout this section, all irreducible polynomials that we work with are assumed to be separable. We will make frequent use of the following lemma, which allows us to furnish monic, separable, irreducible polynomials with a prescribed penultimate coefficient.

\begin{lem}\label{lem:irreducible}
Let $k$ be a field. Pick $a\in k$. Assume that $k$ admits a finite separable extension of degree $n\geq 2$. If $\Char{k}=n=2$, we assume that $a\neq 0$. Then there is a monic, separable, irreducible polynomial $f(t)\in k[t]$ of degree $n$ whose degree $n-1$ coefficient is $a$.
\end{lem}
\begin{proof}
When $k$ is a finite field, this is a special case of the Hansen--Mullen conjecture, which was proved by Wan~\cite{Wan97} and Ham--Mullen~\cite{HM98}. We may thus assume that $k$ is an infinite field, although we will not need this assumption in most cases. In general, the assumption that $k$ admits a finite separable extension of degree $n$ implies that there is a monic, separable, irreducible polynomial $m(t)\in k[t]$ of degree $n$. The goal is to use $m(t)$ to find another monic, separable, irreducible polynomial $f(t)$ with the prescribed coefficient in degree $n-1$.

To begin, assume that $\Char{k}=0$ or that $\Char{k}$ does not divide $n$. 
\begin{enumerate}[(i)]
\item If the degree $n-1$ coefficient of $m(t)$ is $c\neq 0$ and the prescribed coefficient is $a\neq 0$, then set $f(t):=(a^{-1}c)^{-n}\cdot m(a^{-1}ct)$. The separability and irreducibility of $f(t)$ follow from that of $m(t)$, and the degree $n-1$ coefficient of $f(t)$ is $(a^{-1}c)^{-n}\cdot c(a^{-1}c)^{n-1}=a$, as desired.
\item If $c=0$ and $a\neq 0$, then set $g(t):=m(t+1)$ (which is again separable and irreducible). By the binomial theorem, the degree $n-1$ coefficient of $g(t)$ is $\binom{n}{1}=n$ (or $n\mod\Char{k}$ in positive characteristic), which is non-zero since we have assumed that $\Char{k}=0$ or $\Char{k}\nmid n$. We can then set $f(t)=(a^{-1}n)^{-n}\cdot g(a^{-1}nt)$ as in (i).
\item If $c\neq 0$ and $a=0$, then set $f(t):=m(t-\frac{c}{n})$. Since $-c$ is the sum of the roots of $m(t)$, the sum of the roots of $f(t)$ is $-c+n\cdot\frac{c}{n}=0$, which is the desired degree $n-1$ coefficient.
\item If $a=c=0$, then we simply take $f(t):=m(t)$.
\end{enumerate}

Now suppose that $\Char{k}=p$ and $n=pq$ for some integer $q>0$. In characteristic $p$, an irreducible polynomial $m(t)\in k[t]$ is separable if and only if it is not of the form $m(t)=h(t^p)$ for some polynomial $h(t)\in k[t]$. Let $m(t)\in k[t]$ be a monic, separable, irreducible polynomial of degree $n$. Then there exists $1\leq d\leq n$ with $p\nmid d$ such that the degree $d$ term of $m(t)$ is non-zero. Let $c$ be the degree $d$ coefficient of $m(t)$.
\begin{enumerate}[(i)]
\setcounter{enumi}{4}
\item If $d=n-1$ and the prescribed coefficient is $a\neq 0$, then set $f(t):=(a^{-1}c)^{-n}\cdot m(a^{-1}ct)$ as in (i).
\item If $d=1$ and $a\neq 0$, then the scaled reciprocal polynomial $m^*(t):=m(0)^{-1}t^n\cdot m(t^{-1})$ is again separable and irreducible with degree $n-1$ coefficient $m(0)^{-1}c\neq 0$. We then set $f(t):=(a^{-1}m(0)^{-1}c)^{-n}\cdot m^*(a^{-1}m(0)^{-1}ct)$.
\item If the degree 1 and $n-1$ terms of $m(t)$ are zero and $a\neq 0$, then consider $g_x(t):=m(t+x)$. The degree 1 term of $g_x(t)$ is $c_x:=\sum_{i=2}^{n-2} im_ix^{i-1}$, where $m_i$ is the degree $i$ coefficient of $m(t)$. Since $c_x$ is a degree $n-2$ polynomial in $x$, our assumption that $|k|=\infty>n-2$ implies that there exists $\alpha\in k$ such that $c_\alpha\neq 0$. It follows that the scaled reciprocal polynomial $g_\alpha^*(t)$ is a monic, separable, irreducible polynomial with degree $n-1$ coefficient $m(\alpha)^{-1}c_\alpha\neq 0$, so we can set $f(t):=(a^{-1}m(\alpha)^{-1}c_\alpha)^{-n}\cdot g_\alpha^*(a^{-1}m(\alpha)^{-1}c_\alpha t)$.
\item If $c\neq 0$ and $a=0$, then let $\alpha$ be a root of $m(t)$. It suffices to find $\beta\in k(\alpha)$ with trace zero such that $\beta$ is not contained in any proper subextension of $k(\alpha)/k$. Once we have done so, it will follow that $k(\beta)=k(\alpha)$ is separable over $k$, and hence the minimal polynomial of $\beta$ will be a monic, separable, irreducible polynomial of degree $n$. By picking $\beta$ with trace zero, the trace of its minimal polynomial (i.e. the degree $n-1$ coefficient) will be zero as well.

The trace defines a $k$-linear map $\mathrm{tr}:k(\alpha)\to k$, so $\mathrm{ker}(\mathrm{tr})$ is a $k$-vector space of dimension $n-1$. Since $k(\alpha)/k$ is a separable extension, there are only finitely many subextensions $k\subset L\subset k(\alpha)$. Moreover, since $[k(\alpha):k]=[k(\alpha):L]\cdot[L:k]$, each subextension $L$ is a vector subspace of $k(\alpha)$ of dimension at most $n/2$. It suffices to choose $\beta\in\mathrm{ker}(\mathrm{tr})-\bigcup_{k\subset L\subset k(\alpha)}L$. If $n>2$, then a finite union of at most $n/2$-dimensional subspaces cannot cover an $(n-1)$-dimensional subspace. Paired with the assumption that $k$ is infinite, it follows that $\mathrm{ker}(\mathrm{tr})-\bigcup_{k\subset L\subset k(\alpha)}L$ is non-empty if $n>2$. The $n=2$ case is already solved by case (iii) and our assumption that $a\neq 0$ if $\Char{k}=n=2$.
\item If $a=c=0$, then we take $f(t):=m(t)$ as in case (iv).\qedhere
\end{enumerate}
\end{proof}

\begin{lem}[$1_6$]\label{lem:27 lines}
Let $k$ be a field with $|k|\geq 16$. Then there is a smooth cubic surface over $k$ with 27 $k$-rational lines.
\end{lem}
\begin{proof}
It suffices to find 6 $k$-rational points on the cuspidal cubic $C$ that are in general position. If $\Char{k}$ is 0 or at least 17, then $G(t)=\prod_{i=1}^6 (t+i)$ satisfies the desired criteria. Otherwise, write $\Char{k}=p$.
\begin{itemize}
\item If $5\leq p\leq 13$, then there exists $\alpha\in k-\mb{F}_p$, and $G(t)=(t+1)(t+2)(t+\alpha)(t+\alpha+1)(t+\alpha+2)(t+\alpha+3)$ satisfies the desired criteria.
\item If $p=3$, then there exist $\alpha,\beta\in k-\mb{F}_3$ such that $\{1,\alpha,\beta\}$ are $\mb{F}_3$-linearly independent. Indeed, if $k$ is a finite extension of $\mb{F}_3$, then $|k|\geq 27$ and hence $\dim_{\mb{F}_3}{k}\geq 3$. If $k$ is an infinite extension of $\mb{F}_3$, then take $\alpha\in k$ to be transcendental over $\mb{F}_3$ and set $\beta=\alpha^2$. In either case, $G(t)=(t+1)(t+2)(t+\alpha)(t+\alpha+1)(t+\beta)(t+\beta+1)$ satisfies the desired criteria.
\item If $p=2$, then there exist $\alpha,\beta,\gamma\in k-\mb{F}_2$ such that $\{1,\alpha,\beta,\gamma\}$ are $\mb{F}_2$-linearly independent. Indeed, if $k$ is a finite extension of $\mb{F}_2$, then $|k|\geq 16$ and hence $\dim_{\mb{F}_2}{k}\geq 4$. If $k$ is an infinite extension of $\mb{F}_2$, then take $\alpha$ to be transcendental over $\mb{F}_2$ and set $\beta=\alpha^2$ and $\gamma=\alpha^3$. In either case, $G(t)=t(t+1)(t+\alpha)(t+\beta)(t+\gamma)(t+\alpha+\beta+\gamma)$ satisfies the desired criteria.\qedhere
\end{itemize}
\end{proof}

\begin{lem}[$1_4,2_1$]\label{lem:15 lines}
Let $k$ be a field with $|k|\geq 8$. Assume that $k$ admits a separable degree 2 extension. Then there is a smooth cubic surface over $k$ with 15 $k$-rational lines.
\end{lem}
\begin{proof}
Let $m(t)=t^2+at+b\in k[t]$ be a separable irreducible polynomial. It suffices to pick four distinct elements $r_1,r_2,r_3,r_4\in k^\times$ such that $\sum_{i=1}^4 r_i\neq a$ and $\sum_{i\neq j}r_i\neq 0$ for $1\leq j\leq 4$ (since the roots of $m(t)$ are not defined over $k$ and are thus not equal to $-r_i-r_j$). Once we have done so, $G(t)=m(t)\cdot\prod_{i=1}^4(t-r_i)$ will satisfy the desired criteria. (The requirement that $r_1,\ldots,r_4\neq 0$ is explained in Remark~\ref{rem:not 0}.)

Pick distinct $r_1,r_2\in k^\times$. We then need to choose $r_3\in k^\times-\{r_1,r_2\}$ such that $r_3\neq -r_1-r_2$, so we may freely pick $r_3\in k^\times-\{r_1,r_2,-r_1-r_2\}$. Finally, we need to choose $r_4\in k^\times-\{r_1,r_2,r_3,-r_1-r_2,-r_1-r_3,-r_2-r_3\}$. Since we have assumed $|k|\geq 8$, such an $r_4$ exists, and we are done.
\end{proof}

\begin{lem}[$1_3,3_1$]\label{lem:9 lines}
Let $k$ be a field with $|k|\geq 7$. Assume that $k$ admits a separable degree 3 extension. Then there is a smooth cubic surface over $k$ with 9 $k$-rational lines.
\end{lem}
\begin{proof}
By Lemma~\ref{lem:irreducible}, there is a separable irreducible polynomial $m(t)=t^3+at^2+bt+c\in k[t]$ with $a\neq 0$. Since $a\neq 0$, the three roots of $m(t)$ do not sum to zero. Since $m(t)$ is irreducible of degree 3, no two roots of $m(t)$ sum to an element of $k$ (or else the third root would belong to $k$ and $m(t)$ would not be irreducible). It thus suffices to find distinct $r_1,r_2,r_3\in k^\times$ such that $r_1+r_2+r_3\neq 0$ and $r_1+r_2+r_3\neq a$. (The requirement that $r_1,r_2,r_3\neq 0$ is explained in Remark~\ref{rem:not 0}.) Once we have done so, $G(t)=m(t)\cdot\prod_{i=1}^3(t-r_i)$ will satisfy the desired criteria.

Pick distinct $r_1,r_2\in k^\times$. We may then freely pick $r_3\in k^\times-\{r_1,r_2,-r_1-r_2,a-r_1-r_2\}$. Since we have assumed $|k|\geq 7$, such an $r_3$ exists.
\end{proof}

\begin{lem}[$1_2,2_2$]\label{lem:7 lines}
Let $k$ be a field with $|k|\geq 5$. Assume that $k$ admits a separable degree 2 extension. Then there is a smooth cubic surface over $k$ with 7 $k$-rational lines.
\end{lem}
\begin{proof}
Let $m(t)=t^2+at+b\in k[t]$ be a separable irreducible polynomial with $a\neq 0$ (which exists by Lemma~\ref{lem:irreducible}). First suppose $\Char{k}\neq 2$, so that $a\neq -a$. Then $n(t):=m(-t)=t^2-at+b$ also does not have any $k$-rational roots, so $n(t)$ is also separable and irreducible over $k$. Pick $r_1\in k^\times-\{\pm a\}$ and $r_2\in k^\times-\{\pm a,\pm r_1\}$, which is possible since $|k|\geq 5$. Then $G(t)=(t-r_1)(t-r_2)\cdot m(t)\cdot n(t)$ satisfies the desired criteria. Indeed, the degree 5 term is $a-a-r_1-r_2=-(r_1+r_2)$, which is non-zero by our choice of $r_2$. The sum of one (respectively, two) rational roots of $G(t)$ with two (respectively, one) non-rational roots of $G(t)$ cannot be zero (since $r_1,r_2\neq\pm a$). Finally, all four roots of $m(t)$ and $n(t)$ sum to zero, so no three of these roots can sum to zero (or else the fourth root would be $0\in k$, a contradiction).

If $\Char{k}=2$, then pick $a\in k-\mb{F}_2$ (which we may do since $|k|\geq 5$). By Lemma~\ref{lem:irreducible}, there exists $b$ such that $m(t)=t^2+at+b$ is separable and irreducible. Take $n(t):=m(t+1)=t^2+at+a+b+1$ as our second irreducible polynomial, and note that $m(t)\neq n(t)$ since $a\neq 1$. Now pick $c\in k^\times-\{1,a\}$ (which we may do since $|k|\geq 5$). Then $G(t)=(t+1)(t+c)\cdot m(t)\cdot n(t)$ satisfies the desired criteria. Indeed, the degree 5 term is $1+c\neq 0$, sums of three roots involving one or two rational roots cannot be zero, and the sum of all four non-rational roots is $a+a=0$, so any three of these roots cannot sum to zero.
\end{proof}

\begin{lem}[$1_2,4_1$]\label{lem:5 lines}
Let $k$ be a field with $|k|\geq 11$. Assume that $k$ admits a separable degree 4 extension. Then there is a smooth cubic surface over $k$ with 5 $k$-rational lines.
\end{lem}
\begin{proof}
Let $m(t)=t^4+at^2+bt+c\in k[t]$ be a separable irreducible polynomial, which exists by Lemma~\ref{lem:irreducible}. Since the sum of the roots of $m(t)$ is zero, no three roots of $m(t)$ can sum to zero. Let $s_1,\ldots,s_6$ be the $\binom{4}{2}$ sums of pairs of roots of $m(t)$. We want to pick $r_1,r_2\in k^\times$ such that $r_1,r_2\neq s_i$ for each $i$. There are at most six elements to avoid (in the case that each $s_i\in k$ and all are distinct). We thus pick $r_1\in k^\times-\{s_1,\ldots,s_6\}$ and $r_2\in k^\times-\{\pm r_1,s_1,\ldots,s_6\}$. Since $|k|\geq 11$, such a choice of $r_1,r_2$ is possible. Now $G(t)=(t-r_1)(t-r_2)\cdot m(t)$ satisfies the desired criteria, since the sum of all roots is $-(r_1+r_2)\neq 0$ and any sum of three roots involving $r_1$ or $r_2$ cannot be zero.
\end{proof}

\begin{lem}[$1_1,2_1,3_1$]\label{lem:3 lines, skew}
Let $k$ be a field with $|k|\geq 4$. Assume that $k$ admits separable extensions of degree 2 and 3. Then there is a smooth cubic surface over $k$ with 3 $k$-rational lines that are skew.
\end{lem}
\begin{proof}
Pick $a\in k^\times-\{-1\}$. By Lemma~\ref{lem:irreducible}, there are separable irreducible polynomials $m(t)=t^3+at^2+bt+c\in k[t]$ and $n(t)=t^2+t+d\in k[t]$. Since $a\neq 0$, the three roots of $m(t)$ do not sum to zero. Note that no three of the five roots of $m(t)$ and $n(t)$ sum to zero. Indeed, the roots of $n(t)$ sum to $-1$, and no root of $m(t)$ can be an element of $k$. If two roots of $m(t)$ and one root of $n(t)$ sum to zero, then the remaining roots $r_m,r_n$ of $m(t),n(t)$, respectively, sum to $-a-1$. But this would imply that $r_m=-r_n-a-1$ is a root of the irreducible polynomial $n(-t-a-1)\in k[t]$, which in turn implies that $n(-t-a-1)$ must be an irreducible factor of $m(t)$. Since $m(t)$ is irreducible, such a factor does not exist.

Finally, take $s\in k^\times-\{1,-a-1,\}$. Then $G(t)=(t+s)\cdot m(t)\cdot n(t)$ satisfies the desired criteria. Indeed, the degree 5 term is $a+s+1\neq 0$. Moreover, no three roots of $G(t)$ sum to zero. We have already seen that no three roots of $m(t)\cdot s(t)$ sum to zero, and no two roots of $m(t)$ can sum to an element of $k$. The remaining possibility is that the sum of the roots of $n(t)$ is $s$, but we have chosen $s\neq 1$.
\end{proof}

\begin{lem}[$2_3$]\label{lem:3 lines}
Let $k$ be a field with $|k|\geq 5$. Assume that $k$ admits a separable extension of degree 2. Then there is a smooth cubic surface over $k$ with 3 $k$-rational lines that are coplanar.
\end{lem}
\begin{proof}
First assume $\Char{k}=2$. Pick a separable irreducible polynomial $m(t)=t^2+t+a\in k[t]$ (which exists due to Lemma~\ref{lem:irreducible}). Now pick distinct $\alpha,\beta\in k-\{0,1\}$ such that $\alpha\neq\beta+1$ (which exist since $|k|\geq 5$), and set $n(t)=m(t+\alpha)=t^2+t+a+\alpha^2+\alpha$ and $p(t)=m(t+\beta)=t^2+t+a+\beta^2+\beta$. Since $\alpha,\beta\neq 0,1$, we have $\alpha^2+\alpha\neq 0$ and $\beta^2+\beta\neq 0$. Since $\alpha\neq\beta$ and $\alpha\neq\beta+1$, we have that $(\alpha,\beta)$ is not a solution to $(x+y)(x+y+1)=0$ and hence $\alpha^2+\alpha\neq\beta^2+\beta$. In particular, the polynomials $m(t),n(t)$, and $p(t)$ are all distinct, so their six collective roots must also be distinct. Moreover, the sum of the four roots of any two of $m(t),n(t),p(t)$ is zero, so no three of these roots can sum to zero. The sum of all six roots is 1, so it remains to check that the sum of three roots, one from each of $m(t),n(t),p(t)$, cannot be zero. Then $G(t)=m(t)\cdot n(t)\cdot p(t)$ will satisfy the desired criteria.

Let $\mu_1,\mu_2$ be the roots of $m(t)$. The roots of $n(t)$ are $\mu_1+\alpha,\mu_2+\alpha$, and the roots of $p(t)$ are $\mu_1+\beta,\mu_2+\beta$. The sum of three roots, one from each of $m(t),n(t),p(t)$, is therefore either $3\mu_i+\alpha+\beta=\mu_i+\alpha+\beta$ for $i\in\{1,2\}$ or $2\mu_i+\mu_j+\alpha+\beta=\mu_j+\alpha+\beta$ for $\{i,j\}=\{1,2\}$. Since $\alpha,\beta\in k$ and $\mu_i,\mu_j\not\in k$, these sums are non-zero.

Now assume $\Char{k}\neq 2$. If $\Char{k}\neq 3$, let $a=2$. If $\Char{k}=3$, pick $a\in k-\mb{F}_3$. Then there is a separable irreducible polynomial $m(t)=t^2+at+b\in k[t]$. Pick $\gamma\in k-\{0,\pm 1,-\frac{a}{2}-1\}$, and set $n(t)=m(t+1)$ and $p(t)=m(t+\gamma)$. Since $\gamma\neq 0,1$, the polynomials $m(t),n(t),p(t)$ are distinct. The sum of the four roots of any two of $m(t),n(t),p(t)$ is $k$-rational, so no three of these roots can sum to zero (or else the fourth root would be rational). The sum of all six roots is $-a-2-2\gamma\neq 0$ (by our choice of $\gamma$), so it remains to check that the sum of three roots, one from each of $m(t),n(t),p(t)$, cannot be zero. Then $G(t)=m(t)\cdot n(t)\cdot p(t)$ will satisfy the desired criteria.

Let $\mu_1,\mu_2$ be the roots of $m(t)$. The roots of $n(t)$ are $\mu_1-1,\mu_2-1$, and the roots of $p(t)$ are $\mu_1-\gamma,\mu_2-\gamma$. The sum of three roots, one from each of $m(t),n(t),p(t)$, is either $3\mu_i-1-\gamma$ or $2\mu_i+\mu_j-1-\gamma=\mu_i-a-1-\gamma$. We have assumed that $\gamma\neq -1$, so $3\mu_i-1-\gamma$ is not zero even in characteristic 3. Since $1,a,\gamma\in k$ and $\mu_1,\mu_2\not\in k$, it follows that neither of these sums can be zero.
\end{proof}

\begin{lem}[$1_1,5_1$]\label{lem:2 lines}
Let $k$ be a field. Assume that $k$ admits a separable extension of degree 5. Then there is a smooth cubic surface over $k$ with 2 $k$-rational lines.
\end{lem}
\begin{proof}
Pick a separable irreducible polynomial $m(t)=t^5+at^3+bt^2+ct+d$. We will set $G(t)=(t+1)\cdot m(t)$. If three roots of $m(t)$ sum to 0, then so do the remaining two roots of $m(t)$. If two roots of $m(t)$ sum to 1, then the roots of $G(t)$ will sum to 0. It thus suffices to show that no two roots of $m(t)$ can sum to 0 or 1.

In characteristic 2, two roots of $m(t)$ summing to zero implies that $m(t)$ has a repeated root, which contradicts our assumption that $m(t)$ is separable. If $\Char{k}\neq 2$ and $r$ is a root of $m(t)$, then $0=m(r)+m(-r)=2br^2+2d$. If $b\neq 0$, then the minimal polynomial of $r$ has degree less than 5, contradicting the irreducibility of $m(t)$. Otherwise, we have $2d=0$, which implies that 0 is a root of $m(t)$ and again contradicts the irreducibility of $m(t)$. If two roots of $m(t)$ sum to 1, then there is a root $r$ such that $m(r)=m(1-r)=0$. Thus $r$ is a root of
\begin{align*}
m(t)+m(1-t)&=5t^4-10t^3+(10+3a+2b)t^2\\
&-(5+3a+2b+c)t+1+a+b+c+2d.
\end{align*} 
If $\Char{k}\neq 5$, then $m(t)+m(1-t)$ is a non-zero polynomial of degree strictly less than 5. This again contradicts the irreducibility of $m(t)$. If $\Char{k}=5$, then $m(t)+m(1-t)$ is identically zero only if $b=-\frac{3a}{2}=a$, $c=0$, and $d=-\frac{1+a+b+c}{2}=\frac{2+a}{4}=3+4a$. If $m(t)+m(1-t)$ is identically zero, then take $G(t)=(t+1)\cdot m(t+1)$. Since $m(t+1)=t^5+at^3+4at^2+a+4$, we follow the previous arguments to find that no two roots of $m(t+1)$ sum to 0 or 1, as desired.
\end{proof}

\begin{lem}[$2_1,4_1$]\label{lem:1 line}
Let $k$ be a field with $|k|\geq 8$. Assume that $k$ admits separable extensions of degree 2 and 4. Then there is a smooth cubic surface over $k$ with 1 $k$-rational line.
\end{lem}
\begin{proof}
Pick $a\in k^\times$. Let $m(t)=t^2+at+b$ and $n(t)=t^4+ct^2+dt+e$ be separable irreducible polynomials. Since all four roots of $n(t)$ sum to zero, no three of these roots can sum to zero. All six roots of $m(t)$ and $n(t)$ sum to $-a\neq 0$. Moreover, the two roots of $m(t)$ sum to $-a\in k$, so no root of $n(t)$ can yield zero when summed with the roots of $m(t)$. It remains to show that a root of $m(t)$ and two roots of $n(t)$ cannot sum to zero. We then set $G(t)=m(t)\cdot n(t)$.

Let $\nu_1,\ldots,\nu_4$ be the roots of $n(t)$. We want to guarantee that each of the roots of $m(t)$ are not of the form $-(\nu_i+\nu_j)$ for some $i,j$. If the splitting field of $m(t)$ is not a subfield of $k(\nu_1)$, then we are done. Otherwise, let $r_{i,j}\in k$ be the trace of $-(\nu_i+\nu_j)$, so that the minimal polynomial of $-(\nu_i+\nu_j)$ has penultimate coefficient $-r_{i,j}$. If we pick $a\in k^\times-\bigcup_{1\leq i<j\leq 6}\{-r_{i,j}\}$ (which is possible since $|k|\geq 8$), then a root of $m(t)$ and two roots of $n(t)$ cannot sum to zero.
\end{proof}

\begin{lem}[$3_2$]\label{lem: 0 lines, 3-3}
Let $k$ be a field with $|k|\geq 4$. Assume that $k$ admits a separable extension of degree 3. Then there is a smooth cubic surface over $k$ with no $k$-rational lines.
\end{lem}
\begin{proof}
First assume $\Char{k}$ is not 2 or 3. Pick separable irreducible polynomials $m(t)=t^3+2t^2+at+b$ and $n(t)=t^3+t^2+ct+d$. By design, the three roots of $m(t)$ and $n(t)$, respectively, do not sum to zero, and all six roots sum to $-3\neq 0$. It remains to show that one root of $m(t)$ and two roots of $n(t)$ do not sum to zero (with the same argument holding for two roots of $m(t)$ and one root of $n(t)$). Once we have done so, $G(t)=m(t)\cdot n(t)$ will satisfy the desired criteria.

Let $\mu_1,\mu_2,\mu_3$ and $\nu_1,\nu_2,\nu_3$ be the roots of $m(t)$ and $n(t)$, respectively. Assume $\mu_1+\nu_1+\nu_2=0$. Then $\nu_3=\mu_1+\sum_{i=1}^3\nu_i=\mu_1+1$. The minimal polynomial of $\mu_1+1$ is $m(t-1)$, so this implies that $m(t-1)=n(t)$. But the degree 2 coefficient of $m(t-1)$ is $2-3=-1$, and the degree 2 coefficient of $n(t)$ is 1. These are not equal when $\Char{k}\neq 2$.

Now assume $\Char{k}=p$ is 2 or 3. Since $|k|\geq 4$, we may pick $\alpha\in k-\mb{F}_p$. Pick separable irreducible polynomials $m(t)=t^3+\alpha t^2+at+b$ and $n(t)=t^3+t^2+ct+d$. Again, the three roots of each of these polynomials do not sum to zero, and their six roots sum to $-\alpha-1\neq 0$. Using the same notation as before, if $\mu_1+\nu_1+\nu_2=0$, then $\nu_3=\mu_1+1$. The minimal polynomial of $\mu_1+1$ is $m(t-1)=n(t)$, but the respective degree 2 coefficients are then $\alpha-3$ and 1. Since $\alpha\in k-\mb{F}_p$, it follows that $\alpha-3\neq 1$.
\end{proof}

\begin{lem}[$6_1$]\label{lem: 0 lines, 6}
Let $k$ be a field with $|k|\geq 23$. Assume that $k$ admits a separable extension of degree 6. Then there is a smooth cubic surface over $k$ with no $k$-rational lines.
\end{lem}
\begin{proof}
First, assume that $\Char{k}\neq 3$. Let $m(t)=t^6+t^5+a_4t^4+\ldots+a_0\in k[t]$ be a separable irreducible polynomial with roots $r_1,\ldots,r_6$. If no three of the roots of $m(t)$ sum to zero, then we set $G(t)=m(t)$. Otherwise, consider $m(t-\alpha)$ for some $\alpha\in k^\times$, whose roots are given by $r_1+\alpha,\ldots,r_6+\alpha$. If $r_1+r_2+r_3=0$, then $(r_1+\alpha)+(r_2+\alpha)+(r_3+\alpha)=3\alpha\neq 0$. However, it may happen that some other roots satisfy $r_i+r_j+r_\ell+3\alpha=0$. If this happens, we take $\beta\in k-\{0,\alpha\}$ and investigate $m(t-\beta)$. There are $\binom{6}{3}=20$ sums of triples of roots to consider, so it appears that $|k|\geq 21$ suffices for our purposes. In fact, there is one more case to avoid: since $\sum_{i=1}^6 r_i=-1$, we have $\sum_{i=1}^6(r_i+\frac{1}{6})=0$. By assuming $|k|\geq 23$, we guarantee that there exists $\alpha\in k-\{-\frac{1}{6}\}$ such that $m(t-\alpha)$ is separable and irreducible, has no three roots summing to zero, and has all six roots not summing to zero. We then set $G(t)=m(t-\alpha)$.

Now assume $\Char{k}=3$. Since $|k|\geq 23$, we have $|k|\geq 27$ in characteristic 3. Let $m(t)=t^6+a_5t^5+\ldots+a_0\in k[t]$ be a separable irreducible polynomial with $a_5\neq 0$. As before, let $r_1,\ldots,r_6$ be the roots of $m(t)$. If no three roots of $m(t)$ sum to zero, then we are done. Otherwise, we will work with the scaled reciprocal of $m(t)$ (possibly after shifting). If $a_1\neq 0$, let $m^*(t):=a_0^{-1}t^6\cdot m(t^{-1})$ be the scaled reciprocal of $m(t)$. If $a_1=0$, then note that the degree 1 coefficient of $m(t-\alpha)$ is $2a_5\alpha^4-a_4\alpha^3-2a_2\alpha=\alpha(2a_5\alpha^3-a_4\alpha^2-2a_2)$. The sum of any three roots of $m(t-\alpha)$ is of the form $(r_i+\alpha)+(r_j+\alpha)+(r_\ell+\alpha)=r_i+r_j+r_\ell$ (since we are in characteristic 3). By choosing $\alpha\in k^\times$ not a root of $2a_5t^3-a_4t^2-2a_2$, we may assume that the linear term of $m(t)$ is non-zero, so that the degree 5 coefficient of $m^*(t)$ is non-zero.

The roots of $m^*(t)$ are $r_1^{-1},\ldots,r_6^{-1}$. The assumption that $r_1+r_2+r_3=0$ implies that $r_1=-r_2-r_3$, so
\begin{align*}
(r_1r_2r_3)(r_1^{-1}+r_2^{-1}+r_3^{-1})&=r_1r_2+r_1r_3+r_2r_3\\
&=-(r_2+r_3)^2+r_2r_3\\
&=-(r_2^2+r_2r_3+r_3^2)\\
&=-(r_2-r_3)^2.
\end{align*}
Since $r_2\neq r_3$ by the separability of $m(t)$, we deduce that $r_1^{-1}+r_2^{-1}+r_3^{-1}\neq 0$. If $r_i^{-1}+r_j^{-1}+r_\ell^{-1}=0$ for some $i,j,\ell$, then consider the reciprocal polynomial of $m(t-\alpha)$. The sum of any three roots is of the form
\begin{align*}
\frac{1}{r_i+\alpha}+\frac{1}{r_j+\alpha}+\frac{1}{r_\ell+\alpha}&=\frac{r_ir_j+r_ir_\ell+r_jr_\ell+2\alpha(r_i+r_j+r_\ell)+3\alpha}{(r_i+\alpha)(r_j+\alpha)(r_\ell+\alpha)}\\
&=\frac{r_ir_j+r_ir_\ell+r_jr_\ell+2\alpha(r_i+r_j+r_\ell)}{(r_i+\alpha)(r_j+\alpha)(r_\ell+\alpha)}\\
&=\frac{r_ir_j+r_ir_\ell+r_jr_\ell-\alpha(r_i+r_j+r_\ell)}{(r_i+\alpha)(r_j+\alpha)(r_\ell+\alpha)}.
\end{align*}
As we have seen, if $r_i+r_j+r_\ell=0$, then $r_ir_j+r_ir_\ell+r_jr_\ell\neq 0$, and the converse holds as well. It remains to ensure that $r_ir_j+r_ir_\ell+r_jr_\ell\neq\alpha(r_i+r_j+r_\ell)$ when $r_i+r_j+r_\ell\neq 0$. There are $\binom{6}{3}=20$ sums of triples to consider. Because $|k|>24$, we can pick $\alpha\in k^\times-\{\rho:2a_5\rho^3+a_4\rho^2+2a_2=0\}$ such that $r_ir_j+r_ir_\ell+r_jr_\ell\neq\alpha(r_i+r_j+r_\ell)$ for all $i,j,\ell$. Let $G(t)$ be the scaled reciprocal of $m(t-\alpha)$. Our choice of $\alpha$ ensures that the degree 5 coefficient of $G(t)$ is non-zero and that no three roots of $G(t)$ sum to zero, as desired.
\end{proof}

\section{Some fields with all line counts}\label{sec:specific}
Using Theorem~\ref{thm:main}, we can understand the set of line counts for smooth cubic surfaces over a given field by looking at the field's Galois theory. For example, since finite fields admit finite (separable) extensions of arbitrary degree, every line count must be realized over finite fields of cardinality at least 23 (reproving Loughran and Trepalin's classification of line counts over finite fields in this range \cite{LT19}). In order to prove Corollary~\ref{cor:fin-gen}, it suffices to show that finitely generated fields and finite transcendental extensions of arbitrary fields each admit separable extensions of arbitrary degrees.

\begin{lem}
Let $k$ be a finitely generated field or a finite transcendental extension of another field. Then for each integer $n>0$, there exists a finite separable extension $k'$ of $k$ with $[k':k]=n$.
\end{lem}
\begin{proof}
If $k$ is a finitely generated field, then let $k_0$ be its prime field (i.e.~$\mb{Q}$ if $\Char{k}=0$ and $\mb{F}_p$ if $\Char{k}=p$). If $k/k_0$ is finite, then $k$ is a number field in characteristic 0 or of the form $\mb{F}_q$ in positive characteristic. In the latter case, take $k'=\mb{F}_{q^n}$. In the former case, let $\mc{O}$ be the ring of integers of $k$, and let $u\in\mc{O}$ be an irreducible non-zero non-unit (such as the uniformizer of a prime ideal). Then there is no element $s\in\mc{O}$ such that $s^2=u$, so $m(t)=t^n+ut+u$ is irreducible in $k[t]$ by Eisenstein's criterion and Gauss's lemma. It thus suffices to set $k'$ to be the splitting field of $m(t)$.

Now suppose $k/k_0$ is not finite. Since $k$ is finitely generated, there exist generators $z_1,\ldots,z_m$ such that $k=k_0(z_1,\ldots,z_m)$. Since $k/k_0$ is not finite, at least one of $z_1,\ldots,z_m$ is transcendental over $k_0$. By reordering if necessary, we may assume that $z_m$ is transcendental over $k_0(z_1,\ldots,z_{m-1})$. Let $F=k_0(z_1,\ldots,z_{m-1})$. Consider $R=F[z_m]$, which is a UFD (since it is a polynomial ring over a field) whose fraction field is $k$. The assumption that $z_m$ is transcendental over $F$ implies that $R$ is not a field, so we can pick a non-zero non-unit $g\in R$. Let $u$ be an irreducible factor of $g$. Then $m(t)=t^n+ut+u$ is irreducible over $R$ by Eisenstein's criterion and hence irreducible over $k$ by Gauss's lemma. Moreover, $m'(t)$ is not identically zero, so $m(t)$ is separable. The splitting field $k'$ of $m(t)$ is thus a degree $n$ separable extension of $k$.

Finally, if $k$ is a finite transcendental extension of some field $k_0$, then there exist transcendental elements $z_1,\ldots,z_m$ such that $k=k_0(z_1,\ldots,z_m)$. We may thus repeat the arguments of the previous paragraph to obtain the desired extension $k'/k$.
\end{proof}

\appendix
\section{Subgroups of $W(\mathrm{E}_6)$}\label{sec:magma}
The following Magma code, provided to us by Dan Loughran, classifies all possible line counts on a smooth cubic surface over any field by considering all conjugacy classes of subgroups of $W(\mathrm{E}_6)$. This provides a modern proof of Theorem~\ref{thm:lines-over-k}. This code is a variant of the freely available code accompanying~\cite{JL15,BFL19}.

\begin{lstlisting}
R_e6 := RootDatum("E6");
Cox_e6 := CoxeterGroup(R_e6);
we6 := StandardActionGroup(Cox_e6);
list:=SubgroupClasses(we6);

number_of_lines := function(G);
temp:=0;
for O in Orbits(G) do
 if #O eq 1 then
  temp:=temp+1;
 end if;
end for;
return temp;
end function;

for rec in list do
G := rec`subgroup;
print Order(G),number_of_lines(G);
end for;
\end{lstlisting}

\bibliography{rational-lines-cubic}{}
\bibliographystyle{alpha}
\end{document}